\documentclass[a4paper,fleqn]{cas-sc}
\usepackage[numbers]{natbib}
\usepackage{graphicx} 
\usepackage{xcolor}
\usepackage{array}
\usepackage[utf8]{inputenc}
\usepackage[T1]{fontenc}
\usepackage{tikz}
\usepackage{subcaption}

\usepackage{graphicx}
\usepackage{multirow}
\usepackage{amsmath,amssymb,amsfonts,amsthm}
\usepackage{mathrsfs}
\usepackage[title]{appendix}
\usepackage{xcolor}
\usepackage{textcomp}
\usepackage{manyfoot}
\usepackage{booktabs}
\usepackage{algorithm}
\usepackage{algorithmicx}
\usepackage{algpseudocode}
\usepackage{listings}

\newtheorem{theorem}{Theorem}
\newtheorem{lemma}[theorem]{Lemma} 
\newtheorem{proposition}[theorem]{Proposition} 
\newtheorem{corollary}[theorem]{Corollary}
\newtheorem{observation}{Observation}
\newtheorem{remark}{Remark}

\usepackage{tabularx}
\newcolumntype{Y}{>{\raggedright\arraybackslash}X}

\begin{document}
\let\WriteBookmarks\relax
\def\floatpagepagefraction{1}
\def\textpagefraction{.001}
\shorttitle{Metric dimension of Cartesian product of stars}
\shortauthors{Davoodi, Jannesari}

\title[mode=title]{Metric dimension of Cartesian product of stars}

\author{Akbar Davoodi}

\author{Mohsen Jannesari}

\nonumnote{}

\begin{abstract}
The metric dimension of a graph is the minimum number of landmark vertices required so that every vertex can be uniquely identified by its distances to the landmarks. 
This parameter captures the fundamental tradeoff between compact information encoding and unambiguous identification in networked systems. 
In this work, we determine exact value for the metric dimension of the Cartesian product $K_{1,m} \square K_{1,n}$, also known as hub-and-spoke grids, across all values of $m$ and $n$. 
In addition, we present a constructive linear-time algorithm that builds a minimum resolving set, providing both theoretical guarantees and practical feasibility. 
We complement our results with visualization of parameter regimes that illustrate the design space. 
The findings establish design rules for minimizing landmark sensors and support applications in graph-based localization, monitoring networks, and intelligent information systems. 
Our results extend the theory of metric dimension and contribute efficient methods of direct relevance to information science and computational graph theory.
\end{abstract}

\begin{keywords}
Metric dimension \sep
Cartesian product graphs \sep
Resolving sets \sep
Linear-time construction \sep
Graph-based localization\\[.1cm]

\MSC[2020] 05C12; 05C40; 05C90; 68R10
\end{keywords}

\maketitle

\section{Introduction}\label{Sec:Intro}
Graph theory is a fundamental area of discrete mathematics with broad applications, ranging from biology~\cite{pav} and neuroscience~\cite{dav} to business administration~\cite{wag} and transit network design~\cite{Transp}. Graphs offer a natural framework for modeling complex networks, including communication systems, transportation networks, and molecular structures. A key challenge in analyzing these networks is determining the minimal information required to uniquely identify locations within them. This leads to the concept of the \textit{metric dimension}, which quantifies the smallest number of reference points, or \emph{landmarks}, needed to distinguish all vertices in a network based on their distances from these landmarks.

In metrology and instrumentation, \emph{landmarks} translate to physical monitors or reference points (e.g., beacons, sensors, fiducial markers) whose distances to unknown nodes are obtained via time of flight, received-signal strength, hop-count, or interferometric phase. When a system exhibits a hub-and-spoke architecture, such as airline route networks, hierarchical data-center fabrics, or branched lattices in materials, its abstract model is the Cartesian product of two stars. Determining the metric dimension of $K_{1,m}\square K_{1,n}$ therefore answers a concrete measurement question: \emph{what is the minimum number of monitors required to localise every node with zero ambiguity?} Establishing an exact formula, rather than loose bounds, enables instrumentation planners to allocate measurement resources optimally and to quantify identifiability with certainty.

Graph-structured localisation and sensor-placement problems of direct relevance to \emph{Measurement} have recently leveraged factor-graph models and graph-based optimisation \cite{zhang2023factorgraph,huang2025graphpool}.
While prior studies optimise sensor locations for specific instruments or fault-diagnosis objectives \cite{huang2025graphpool,mermer2025entropy}, we provide an \emph{exact} combinatorial baseline, via metric dimension, for the minimum number of monitors, independent of noise models or learning heuristics. 
In addition to identifying the minimum number, we explicitly compute a metric basis for the graph, thereby determining the exact placement of these sensors.

Besides their theoretical interest, our results have direct practical implications. The Cartesian product of star graphs naturally represents various structured sensor network deployments, combining centralized monitoring efficiency with robustness via intermediate relay nodes. The minimal resolving sets explicitly derived here not only minimize sensor deployment and operational costs but also systematically address critical practical issues, including limited communication range, measurement uncertainties, and reliability concerns such as sensor failures or battery depletion. These explicit practical insights substantially extend the value and applicability of our theoretical contributions, bridging the gap between rigorous combinatorial analysis and real-world instrumentation needs.

To formalize this concept, we consider only finite, simple graphs $G = (V, E)$ throughout this paper. The distance between two vertices $u$ and $v$, denoted by $d_G(u, v)$, is the length of the shortest path between them in $G$. For simplicity, we write $d(u, v)$ when no ambiguity arises.

For an ordered subset $W = \{w_1, \ldots, w_k\}$ of $V(G)$ and a vertex $v$ of $G$, the \textit{metric code} or \textit{metric representation} of $v$ with respect to $W$ is defined as
\[ r(v|W) := (d(v, w_1), \ldots, d(v, w_k)). \]
The set $W$ is a \textit{resolving set} for $G$ if distinct vertices of $G$ have different metric representations with respect to $W$. A resolving set $W$ for $G$ with the minimum cardinality is called a \textit{metric basis} of $G$, and its cardinality is the \textit{metric dimension} of $G$, denoted by $\dim(G)$.

The concepts of resolving sets and the metric dimension of a graph were independently introduced by Slater~\cite{Slater1975} and Harary and Melter~\cite{Harary}. Resolving sets have various applications across diverse fields, including coin weighing problems~\cite{Erdos, coin}, network discovery and verification~\cite{net2}, robot navigation~\cite{landmarks}, the mastermind game~\cite{CartesianProduct}, pattern recognition and image processing~\cite{digital}, combinatorial search and optimization~\cite{coin}, unique representation of chemical compounds~\cite{chem, Ollerman, ChemChart}, and error-correcting codes~\cite{code}.

Determining the metric dimension is generally NP-hard~\cite{landmarks}, and it remains NP-complete even for specific families of graphs, such as planar graphs \cite{Diaz}, bipartite graphs, cobipartite graphs, split graphs, and line graphs of bipartite graphs~\cite{Levin}. Researchers have extensively explored the resolving sets of various graph types and network architectures, including complete graphs, paths, trees~\cite{Ollerman}, 2-trees~\cite{2-trees}, generalized fat trees~\cite{FatTree}, star fan graphs~\cite{starfan}, cycles~\cite{ChemChart}, unicycles~\cite{Ollerman}, fans~\cite{caceres2005metric}, wheels~\cite{shanmukha2002metric}, grids~\cite{digital}, Villarceau grids~\cite{VillarceauGrids}, fractal cubic networks~\cite{fractal}, prisms~\cite{javaid2008families}, Petersen graphs~\cite{ahmad2013metric}, Hamming graphs~\cite{tillquist2019low}, Paley graphs~\cite{fijavvz2004rigidity}, Cayley graphs~\cite{fehr2006metric}, Harary graphs~\cite{javaid2008families},  nanocone networks~\cite{carbon}, SiO\textsubscript{2} nanostructures~\cite{Farooq2025}, 
GeSbTe superlattice chemical structures~\cite{Liqin2023}, circulant graphs~\cite{circulant}, and Johnson and Kneser graphs~\cite{bailey2013resolving}. Graphs with $n$ vertices and metric dimension of $n-3$ are characterized in~\cite{n-3}, while randomly $k$-dimensional graphs are studied in~\cite{random}. The metric dimension of the lexicographic product of graphs has been investigated in~\cite{saputro2013metric, lexico}. The corona product has been studied in~\cite{imran2018metric, iswadi2008metric, kuziak2010corrections, kuziak2017computing, yero2011metric}, the strong product in~\cite{rodriguez2015metric, adar2018metric, kuziak2017resolvability}, and the tensor (or direct) product in~\cite{kuziak2017resolvability}. 
In addition to these well-known products, other graph operations such as line graphs have also been considered in the context of metric dimension~\cite{line}. 
Among various graph products, the Cartesian product stands out due to its structural richness, numerous applications, and well-developed theory, particularly in studies related to metric dimension and resolvability \cite{hammack2011handbook}.

The Cartesian product of graphs $G$ and $H$, denoted by $G \square H$, is the graph with vertex set $V(G) \times V(H) := \{(u, v) : u \in V(G), v \in V(H)\}$, where $(u, v)$ is adjacent to $(u', v')$ whenever $u = u'$ and $vv' \in E(H)$, or $v = v'$ and $uu' \in E(G)$. Note that if $G$ and $H$ are both connected, then $G \square H$ is also connected. Moreover, for every $(u, v), (u', v') \in V(G \square H)$, we have $d_{G \square H}((u, v), (u', v')) = d_G(u, u') + d_H(v, v')$. 
One of the earliest results on the metric dimension of Cartesian product graphs, established by Chartrand et al. \cite{Ollerman}, is
$\dim(H)\leq \dim(H\square K_2)\leq \dim(H)+1$.
Later, C\'{a}ceres et al.~\cite{CartesianProduct} advanced this study by determining the metric dimension $G \square H$ for specific graph pairs $G, H \in \{P_n, C_n, K_n\}$. In their investigation, they introduced the concept of the doubly resolving number of a graph, a notion that was further examined in~\cite{doubly}, where graphs with a doubly resolving number of $2$ were fully characterized.

Hernando et al.~\cite{caceres2005metric} showed that for two connected graphs $G$ and $H$, the metric dimension satisfies $$\max\{ \dim(G), \dim(H) \} \leq \dim (G \square H) \leq \min\big\{\dim(G) + | V(H) | , \dim(H) + | V(G) | \big\} - 1.$$ 
They also proved that $\dim (G \square K_n) \leq \dim (G) + n - 2$ for $n \geq 3$. Furthermore, they demonstrated that 
$\dim (G) \leq \dim (G \square P_n) \leq \dim (G) + 1$,
and that
$\dim (G \square C_n) \leq \dim (G) + 1$ 
if $n$ is odd, whereas 
$\dim (G \square C_n) \leq \dim(G) + 2$
if $n$ is even.

There are few results on the exact value of $\dim(G \square H)$ due to the complexity of the problem. These results are mostly limited to the Cartesian product of a few well-known families of graphs. For example, Khuller et al.~\cite{landmarks} showed that $\dim(P_m \square P_n) = 2$, and Hernando et al.~\cite{caceres2005metric} showed that $\dim(C_m \square C_n) = 3$ if $mn$ is odd, and $\dim(C_m \square C_n) = 4$ if $mn$ is even. Additionally, for $m \leq n$, $\dim(K_m \square K_n) = (n - 1)$ if $2(m - 1) < n$, but $\dim(K_m \square K_n) = \left\lfloor \frac{ 2(m+n-1)}{3} \right\rfloor$ otherwise. For $n \geq 3$, $\dim(P_m \square K_n) = (n - 1)$.

In this paper, we investigate the metric dimension of the Cartesian product graph formed by any combination of two star graphs, determining the exact value for two stars of arbitrary sizes. Based on the literature, the best one can conclude for $dim (K_{1,m} \square K_{1,n})$, assuming without loss of generality that $m \leq n$, is that the metric dimension is bounded between $n - 1$ and $m + n - 1$. We go beyond these bounds by providing the precise value of $dim (K_{1,m} \square K_{1,n})$. To present these results accurately, we first need to define some notations, terminology, and tools.

We use the notation $(v_1, v_2, \ldots, v_n)$ for a path graph with $n$ vertices, where $v_i$ is adjacent to $v_{i+1}$ for $1 \leq i \leq n-1$. We denote by $N(v)$ the set of all neighbors of $v$. We refer to the bipartite graph $K_{1, n}$ as a star graph.

One of our tools is the concept of an adjacency resolving set, defined by Jannesari and Omoomi~\cite{lexico}. Let $G$ be a graph and $W = \{w_1, \ldots, w_k\} \subseteq V(G)$. For each vertex $v \in V(G)$, the \textit{adjacency representation} of $v$ with respect to $W$ is the $k$-vector
\[ r_2(v|W) = (a_G(v, w_1), \ldots, a_G(v, w_k)), \]
where $a_G(v, w_i) = \min\{2, d_G(v, w_i)\}$ for $1 \leq i \leq k$. The set $W$ is an \textit{adjacency resolving set} for $G$ if the vectors $r_2(v|W)$ for all $v \in V(G)$ are distinct. The minimum cardinality of an adjacency resolving set is called the \textit{adjacency dimension} of $G$, denoted by $\dim_2(G)$. An adjacency resolving set of cardinality $\dim_2(G)$ is an \textit{adjacency basis} of $G$. A subset $U$ of vertices is adjacency resolved by $S \subseteq V(G)$ if the adjacency representations of all vertices of $U$ with respect to $S$ are different. A straightforward observation is as follows:
\begin{observation}\label{observation}
	$S \subseteq V(G)$ is an adjacency resolving set for $G$ if and only if $N(v) \cap S$ is unique for all $v \in V(G) \setminus S$.
\end{observation}

\section{Results}\label{Sec:Results}

Let $G := K_{1,m} \square K_{1,n}$ be the Cartesian product of two stars with $m$ and $n$ leaves, respectively. Throughout this paper we assume, without loss of generality, that $m\leq n$. Our aim is to determine the exact value of the metric dimension of $G$ for all values of $n$ and $m$, 
by analysing the following three ranges separately:
\[
\text{(i) } m=1,\hspace*{2cm}
\text{(ii) } 2\leq m<\frac{n}{2},\hspace*{2cm}
\text{(iii) } \frac{n}{2}\leq m\leq n.
\]
The corresponding results are stated  in Theorems \ref{thm:SmallCases}, \ref{2<m<n/2}, and \ref{n/2<m}. 
Before turning to the proofs we fix the notation used throughout and record a preliminary observation that will streamline later arguments.

\subsection{Basic notations and grid view}

First, observe that there is a unique vertex in $G$ with degree $m+n$, denoted by $a_{0,0}$. Let $N := \{c_1, \ldots, c_n\}$ be the set of vertices in $G$ with degree $m+1$, and let $M := \{r_1, \ldots, r_m\}$ be the set of vertices in $G$ with degree $n+1$. This arrangement allows us to represent $G$ as a grid, illustrated in Figure~\ref{fig:Graph}. Note that any other vertex in $G$ has degree two and is denoted by $a_{i,j}$, corresponding to its row and column indices. For convenience, let $A:=\{a_{i,j} \mid 1\leq i\leq m, 1\leq j\leq n\}$. We define the $i$th row and the $j$th column of $G$ as  $R_i:=\{r_i\}\cup\{a_{i,j} \mid 1\leq j\leq n\}$ and $C_j:=\{c_j\}\cup\{a_{i,j} \mid 1\leq i\leq m\}$, respectively.

\begin{figure}[h]
	\centering
	\begin{tikzpicture}
		\fill[gray!20] (0,0) rectangle (8,6); 
		\node at (4,3) {
			\includegraphics[width=0.45\linewidth]{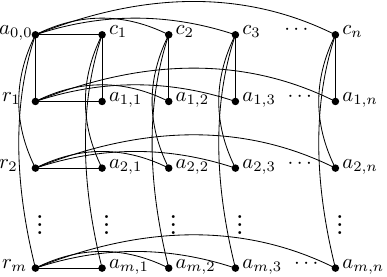}
		};
	\end{tikzpicture}
	\caption{A grid representation of $ K_{1,m} \square K_{1,n}.$}
	\label{fig:Graph}
\end{figure}

\begin{lemma}
	There exists a basis of $ K_{1,m} \square K_{1,n}$ that does not contain $ a_{0,0} $.
\end{lemma}
\begin{proof}
	Let $B$ be a basis of $G:= K_{1,m} \square K_{1,n}$ containing the vertex $a_{0,0}$. Define $B^*:= B\setminus \{a_{0,0}\}$. Since $B^*$ is not a resolving set, there exist two distinct vertices, say $x$ and $y$, such that
	\begin{align}
		r(x | B^*) &= r(y | B^*) \text{, and} \label{eq:first} \\
		d(a_{0,0}, x) &\neq d(a_{0,0}, y). \label{eq:second}
	\end{align}
	Since $d(a_{0,0}, a_{i,j}) = 2$ and $d(a_{0,0}, r_i) = d(a_{0,0}, c_j) = 1$ for every $i,j \geq 1$, we have $\{x, y\} \not\subseteq M \cup N$ and $\{x, y\} \not\subseteq A $. Therefore, by Equation \eqref{eq:second}, there are two possibilities for $x$ and $y$.
	
	\textbf{Case 1:} $a_{0,0} \not\in \{x, y\}$. As one of the vertices $x$ and $y$ belongs to $A$ and the other is in $M \cup N$, without loss of generality, assume that $x \in M \cup N$ and $y \in A$. Since $G$ is not a path graph, we have $|B| \geq 2$, and hence, $B^* \neq \emptyset$.
	
	Suppose $B^*$ contains a vertex of the form $a_{i,j} \in A$. Consider the distance between $a_{i,j}$ and $y \in A$. If $a_{i,j}$ and $y$ lie in a same row or column, then $d(a_{i,j}, y)=2$; otherwise if they lie in different rows and columns, we have $d(a_{i,j}, y)=4$.  Now consider the distance between $a_{i,j}$ and $x\in M\cup N$. If $x$ and $a_{i,j}$ lie in a same column, then $d(a_{i,j}, x)=1$; otherwise, the distance is $d(a_{i,j}, x)=3$.  Therefore, 
	$d(a_{i,j}, y) \in \{2, 4\}$ while $d(a_{i,j}, x) \in \{1, 3\}$,  meaning that $x$ and $y$ are resolved by $a_{i,j} \in B^*$,  contradicting Equation \eqref{eq:first}. 
	
	Thus, we may assume $B^*$ contains a vertex from $M \cup N$, say $w$.  Since $w\in M\cup N$ and $y\in A$, it follows that $d(w, y) \in \{1, 3\}$,  depending on whether $w$ and $y$ are in the same column or not. On the other hand, $d(w, x) \in \{0, 2\}$,  since $x\in M\cup N$. Again  $x$ and $y$ are resolved by $w$, which contradicts Equation \eqref{eq:first}.  This completes Case 1.
	
	\textbf{Case 2:} $a_{0,0} \in \{x, y\}$. Without loss of generality, assume $a_{0,0} = x$.  If $a_{0,0}$ and $y$ are the only two vertices with the same representation with respect to $B^*$, then  the set $B^* \cup \{y\}$ forms a resolving set of $G$ with the same cardinality as $B$, and hence it is also a basis of $G$.  Since this basis excludes $a_{0,0}$, the claim follows.
	
	Otherwise, there is a third vertex $z\not\in\{a_{0,0}, y\}$ such that $r(z|B^*)=r(y|B^*)=r(a_{0,0}|B^*)$. In this case , the pair $y$ and $z$ fall under Case 1, completing the argument.
\end{proof}

In what follows, we investigate the metric dimension of the Cartesian product graphs $ K_{1,m} \square K_{1,n}$, focusing on specific cases that help to better understand the overall behavior of the metric dimension of these graphs. 
Given that $m \leq n$ is always assumed in our analysis, we begin by addressing the simplest case, where $m = 1$ and $n$ is arbitrary. This provides a foundation for understanding how the dimension behaves when the smaller part of the graph is minimal. From there, we extend our analysis to cases where $m$ takes larger values, specifically when $2 \leq m < \frac{n}{2}$, providing a more general result that covers a wider range of possibilities for $m$ and $n$.

\subsection{Case $m=1$}\label{Subsec:m1}
We start with the extremal situation where the smaller factor is very small.  This case is conceptually simple yet already exhibits the flavour of the arguments used later.

\begin{theorem}\label{thm:SmallCases}
	For the graph $ K_{1,1} \square K_{1,n}$, we have
	\[
	\dim( K_{1,1} \square K_{1,n}) = 
	\begin{cases}
		2 & \text{if } n = 1, \\
		n & \text{if } 2 \leq n \leq 4, \\
		n - 1 & \text{if } n \geq 5.
	\end{cases}
	\]
\end{theorem}

\begin{proof}
	Let $G:= K_{1,1} \square K_{1,n}$. If $n = 1$, then $G$ is a cycle of length $4$, and since $\dim(C_4) = 2$, we have $\dim(G) = 2$. 
	
	If $n \geq 2$, we know that $\dim(K_{1,n}) = n - 1$. Moreover, as mentioned in the Introduction, by a result in~\cite{Ollerman}, for any graph $H$, $\dim(H) \leq \dim(H \square K_2) \leq \dim(H) + 1$. Therefore, $\dim(G) \in \{n - 1, n\}$.
	
	For the small cases $2 \leq n \leq 4$, it is straightforward to check that $\dim(G) = n$. Thus, we assume that $n \geq 5$.
	
	We claim that the set $S = \{c_1, c_2, a_{1,3}, a_{1,4}, \ldots, a_{1,n-1}\}$ is a resolving set of size $n - 1$ for $G$. To prove this, it suffices to show that the metric representation of each vertex in $V(G) \setminus S$ is unique.
	
	First, consider $a_{0,0}$, which is the only vertex whose metric code (with respect to $S$) starts with two 1s. Similarly, $r_1$ is the only vertex whose metric code from the third to the $(n-1)$th position consists entirely of 1s. Next, the vertices $a_{1,1}$, $a_{1,2}$, and $a_{1,n-1}$ are the only ones whose metric codes begin with $(1,3)$, $(3,1)$, and $(3,3)$, respectively.
	
	Now, it remains to show that the metric code of each $c_i$ (for $3 \leq i \leq n$) is unique. For each $3 \leq j \leq n$, $c_j$ is the only vertex among the $c_i$'s, where $d(c_j, a_{1,j}) = 1$. Hence, the metric code of each vertex is unique, and the proof is complete.
\end{proof}

\subsection{Case $2\leq m<\frac{n}{2}$}\label{Subsec:SlimRange}
In this range the smaller star is strictly less than half the size of the larger one.  The following theorem shows that almost every column must contribute a landmark.

\begin{theorem}\label{2<m<n/2}
	If $2 \leq m < \frac{n}{2}$, then $\dim( K_{1,m} \square K_{1,n}) = n-1$.
\end{theorem}
\begin{proof}
	Let $G: = K_{1,m} \square K_{1,n}$.
	Note that  $\dim(G) \geq n-1$, otherwise there exists a resolving set  $S$  for $G$ that does not intersect two columns of $G$. If $x,y$ are two distinct vertices in these two columns that are in the same row of $G$, then these two vertices do not resolved by any member of $S$, which is impossible.
	
	Now define 
	$$B := \{a_{1,1}, a_{1,2}, a_{2,3}, a_{2,4}, \ldots, a_{m,2m-1}, a_{m,2m}\} \cup\{a_{m,2m+1}, \ldots, a_{m, n-1}\}.$$
	This set is constructed by selecting two consecutive vertices from each of the $m$ rows, followed (when $n > 2m+1$) by one vertex per column in the remaining columns, all taken from the last row. Note that the last column is excluded. This guarantees that  $|B| = n-1$, each row intersects $B$ in at least two vertices, and every column except the last one intersects $B$.
	We claim that $B$ is a resolving set for $G$. To prove this claim, note that $r_i, 1 \leq i \leq m$, is the only vertex of $G$ where the $(2i-1)$th and $(2i)$th entries of $r(r_i|B)$ are $1$, hence $r(r_i|B)$ is unique. Similarly, for every $j, 1 \leq j \leq n-1$, the metric representation of $c_j$ is unique. Also, $c_n$ is the only vertex of $G$ with a metric representation $(3, 3, \ldots, 3)$, hence the metric representations of the $c_j$'s are unique.
	
	Now let $r(a_{i,j}|B) = r(a_{r,s}|B)$. Note that $r = i$, because the vertices in $A\cap R_i$ are the only vertices with $2$ in the $(2i-1)$th and $2i$th entries of their representations. If $s = n$, then just two entries of $r(a_{r,s}|B)$ are $2$ and the others are $4$. Since vertices in $A\cap C_n$ are the only vertices with this property, $a_{r,s}$ and $a_{i,j}$ are in the $n$th column. If $s \neq n$, then its metric representation has exactly three $2$'s and the other entries are $4$. Clearly, the entries in positions $2i$ and $2i-1$ are $2$, and the third $2$ in the metric representation of $a_{r,s}$ and $a_{i,j}$ specifies the column of these vertices. Since their metric representations are equal, their columns are the same. Therefore, the metric representations of all $a_{i,j}$, $i, j > 0$, are unique. On the other hand, all entries of $r(a_{0,0}|B)$ are $2$ while at most three entries of $r(a_{i,j}|B)$ are $2$.
	
	Since $2 \leq m \leq \frac{n}{2}$, we have $n \geq 5$ and $|B| = n-1 \geq 4$. Therefore, the metric representation of $a_{0,0}$ is unique, and we are done.
\end{proof}

\subsection{Case $\frac{n}{2} \leq m\leq n$}
Theorems \ref{thm:SmallCases} and \ref{2<m<n/2}  settle every instance with $ m <\frac{n}{2}$.
To complete the analysis we investigate the complementary range $\frac{n}{2} \leq m \leq n$. The main result in this regime is the following theorem.

\begin{theorem}\label{n/2<m}
	If ${n\over2}\leq m\leq n$, then $\dim(K_{1,m}\square K_{1,n})=n+\lfloor{2m-n\over 3}\rfloor$.
\end{theorem}

The proof of Theorem~\ref{n/2<m} is given after a sequence of preparatory lemmas that clarify how the relationship between $m$ and $n$ in this specific range affects the metric dimension of the product graph.  In this regime, we also employ an auxiliary bipartite graph that instead of encoding distances within $K_{1,m}\square K_{1,n}$, captures adjacency relations in the new structure.

Let $B$ be a subset of $V(G)$, where  $G:=K_{1,m}\square K_{1,n}$. We define  an auxiliary graph $H := H(G, B)$ with vertex set $V(H) := M \cup N \cup B'$, where $B' := \{b' \mid b \in B\}$ is a  disjoint copy of $B$.  For each vertex $a_{i,j} \in B$,  we add edges $r_i a'_{i,j} \in E(H)$ and $c_j a'_{i,j} \in E(H)$. If $c_i \in B$, then $c_i c'_i \in E(H)$,  and similarly, if $r_i \in B$,  we include the edge $r_i r'_i \in E(H)$.

It is clear that $H$ is a bipartite graph with partite sets $B'$ and $M \cup N$. The degree of every vertex in $B'$ is uniquely determined: the degree of all vertices in $\{x' \in V(H) \mid x \in (M \cup N) \cap B\}$ is $1$, while  all other vertices in $B'$ have degree $2$. A vertex in $M \cup N$  has degree zero if no element of $B$ lies in the corresponding row or column; otherwise,  its degree equals the number of such elements plus one.

The following lemma reveals a key property of this auxiliary structure:

\begin{lemma}\label{B' is adjacency resolving set}
	If $B$ is a resolving set for  $K_{1,m}\square K_{1,n}$, then every pair of vertices in $M \cup N$ have different adjacency representations in  $H := H(K_{1,m}\square K_{1,n}, B)$ with respect to $B'$.
\end{lemma}
\begin{proof}
	Let $G:=K_{1,m}\square K_{1,n}$ and let $x, y \in M \cup N$ and $b \in B$ be such that $d_G(x, b) \neq d_G(y, b)$. If $b \in M \cup N$, then $b \in \{x, y\}$, say $b = x$. Hence, $a_H(x, b') = 0 \neq a_H(y, b')$. Now, let $b = a_{i,j}$ for some $1 \leq i \leq m$ and $1 \leq j \leq n$. Since $d_G(x, b) \neq d_G(y, b)$, $b$ is in a common row or column with exactly one of the vertices $x$ or $y$, say $x$. Thus, $a_H(x, b') = 1$ and $a_H(y, b') = 2$. This means $x$ and $y$ are adjacency resolved by some vertex in $H$ with respect to $B'$.
\end{proof}

To  prepare for the proof of Theorem~\ref{n/2<m}, we need to consider the following cases separately.\\
Case 1: $m$ and $n$ are both at least 5.\\
Case 2: Either $m$ or $n$ is smaller than 5.

We begin with Case 1, assuming without loss of generality that $5 \leq m \leq n$.

\subsection{Case 1: $5 \leq m \leq n$}

\begin{lemma}\label{dim(G)<|S|}
	Let $5 \leq m \leq n$ and  $S \subseteq V(K_{1,m}\square K_{1,n})$. If $S'$, the corresponding set of $S$ in  $H_S := H(K_{1,m}\square K_{1,n}, S)$, adjacency resolves $M \cup N$ in $H_S$, then  $\dim(K_{1,m}\square K_{1,n}) \leq |S|$.
\end{lemma}

\begin{proof}
	The statement clearly holds if $S$ is a resolving set for  $G:=K_{1,m}\square K_{1,n}$. Thus, we may assume that $S$ is not a resolving set for $G$, i.e., there exist $x, y \in V(G) \setminus S$ such that for every $s \in S$, $d(x, s) = d(y, s)$.
	
	We claim that $a_{0,0} \in \{x, y\}$. Otherwise, since $x \neq y$, they must belong to different rows or columns of $G$. Without loss of generality, assume $x \in R_i$ and $y \in R_j$ for some $i \neq j$. By the assumption that all vertices of $H_S$ and specifically $r_i$ and $r_j$ are resolved by $S'$, there exists $s' \in S'$ such that $a(r_i, s') \neq a(r_j, s')$. Without loss of generality, assume $a(r_i, s') < a(r_j, s')$, implying $s' \in R_i$. Then $s$, the corresponding vertex in $G$, resolves $x$ and $y$, which contradicts the initial assumption. Hence, $a_{0,0} \in \{x, y\}$.
	
	Without loss of generality, assume $x = a_{0,0}$. Now, we claim that $y \notin M \cup N$. Otherwise, if $S \cap (M \cup N) = \emptyset$, then for every $s \in S$, $d(x, s) = 2$ but $d(y, s) \in \{1, 3\}$. Also, if $S$ intersects $M \cup N$ at an element $s_0$, then $d(x, s_0) = 1$ while $d(y, s_0) = 2$.
	
	Therefore, we proceed with the proof considering the only possible case where $x = a_{0,0}$ and $y = a_{i,j}$. First, observe that $S \subseteq R_i \cup C_j$. If there exists $s^* \in S \setminus (R_i \cup C_j)$, then $d(y, s^*) - d(x, s^*) = 2$, a contradiction.
	
	By the definition of $H_S$, if $S$ doesn't intersect a row (or column) in $G$, then the corresponding row (or column) is an isolated vertex in $M \cup N$ in $H_S$. Consequently, there is at most one such row or column, given the assumption that $S'$ adjacency resolves $M \cup N$ in $H_S$. Therefore, if $A := S \cap R_i \setminus \{r_i, a_{i,j}\}$ and $B := S \cap C_j \setminus \{c_j, a_{i,j}\}$, then $|A| \geq n-2$ and $|B| \geq m-2$. The assumption $5 \leq m \leq n$ guarantees the existence of indices $1 \leq z, z', z'' \leq m$ and $1 \leq t, t', t'' \leq n$ such that $\{a_{i,t}, a_{i,t'}, a_{i,t''}, a_{z,j}, a_{z',j}, a_{z'',j}\} \subseteq S$. Now, define the set $S_1 := (S \setminus \{a_{i,t}\}) \cup \{a_{z,t}\}$. We claim that $S_1$ is a resolving set for $G$.
	
	First, we show that $S'_1$, the corresponding set in $H_{S_1}$, adjacency resolves $M \cup N$ in $H_{S_1} = H(G, S_1)$. Assume to the contrary that there exist $v, w \in M \cup N$ such that for every $s'_1 \in S'_1$, $a_{H_{S_1}}(v, s'_1) = a_{H_{S_1}}(w, s'_1)$. On the other hand, since $S'$ is an adjacency resolving set for $M \cup N$ in $H_{S}$, there exists $s'_0 \in S'$ such that $a_{H_{S}}(v, s'_0) \neq a_{H_{S}}(w, s'_0)$. Note that for every $q \in S'_1 \cap S'$ and for any $k \in M \cup N$, $a_{H_{S}}(k, q) = a_{H_{S_1}}(k, q)$. Therefore, $a_{H_{S}}(v, s') = a_{H_{S}}(w, s')$ for all $s' \in S' \setminus \{s_0\}$. Consequently, $s'_0 = a'_{i,t}$ and $c_t \in \{v, w\}$. On the other hand, $a_{H_{S_1}}(c_t, a'_{z,t}) = 1$. Thus, $\{v, w\} = \{r_z, c_t\}$. However, $a_{H_{S_1}}(r_z, a'_{zj}) = 1 \neq 2 = a_{H_{S_1}}(c_t, a'_{zj})$. Hence, $S'_1$ adjacency resolves $M \cup N$ in $H_{S_1}$.
	
	For a contradiction, suppose that $S_1$ is not a resolving set for $G$. Then there exist $x_1, y_1 \in V(G)$ such that $d(x_1, s_1) = d(y_1, s_1)$ for every $s_1 \in S_1$. With a similar argument, it can be concluded that one of them, say $x_1 = a_{0,0}$ and $y_1 = a_{i_1,j_1}$, and $S_1 \subseteq R_{i_1} \cup C_{j_1}$. However, by definition, $S_1$ intersects two rows of $G$, which is a contradiction.
\end{proof}

Now, we describe additional properties of a basis $B$ in  $K_{1,m}\square K_{1,n}$ and its corresponding set $B'$ in $H$.

\begin{lemma}\label{P has no leaf in B'}
	Let $5 \leq m \leq n$ and $B$ be a basis of  $K_{1,m}\square K_{1,n}$. If $P$ is an induced maximal path of order at least five in  $H := H(K_{1,m}\square K_{1,n}, B)$, then $P$ has no leaf in $B'$.
\end{lemma}

\begin{proof}
	Suppose, to the contrary, that there exists an induced maximal path $P = (v_1, v_2, v_3, v_4, v_5, \ldots)$ in $H$ with a leaf $v_1 \in B'$. By Lemma~\ref{B' is adjacency resolving set}, $M \cup N$ is adjacency resolved by $B'$. Hence, $N(u) \cap B'$ is unique for each $u \in M \cup N$. We will show that $M \cup N$ is adjacency resolved by $B' \setminus \{v_1\}$.
	
	Deleting $v_1$ can only change the set of neighbors of $v_2$, and $v_2$ can share its neighbors only with $v_4$, because $P$ is an induced maximal path. Therefore, if $M \cup N$ is not adjacency resolved by $B' \setminus \{v_1\}$, then we have $N(v_2) \cap (B' \setminus \{v_1\}) = N(v_4) \cap (B' \setminus \{v_1\})$. However, $v_5 \in (N(v_4) \setminus N(v_2)) \cap (B' \setminus \{v_1\})$. This contradiction implies that $M \cup N$ is adjacency resolved by $B' \setminus \{v_1\}$.
	
	Thus, Lemma~\ref{dim(G)<|S|} implies that  $\dim(K_{1,m}\square K_{1,n}) \leq |B' \setminus \{v_1\}| = \dim(K_{1,m}\square K_{1,n}) - 1$, which is impossible, and we are done.
\end{proof}

\begin{corollary}\label{no even maximal path>4}
	If $5 \leq m \leq n$ and $B$ is a basis of  $K_{1,m}\square K_{1,n}$, then every induced maximal even path in  $H := H(K_{1,m}\square K_{1,n}, B)$ has order at most $4$.
\end{corollary}

\begin{proof}
	Note that $H$ is a bipartite graph and every even path in $H$ has a leaf in $B'$. Therefore, Lemma~\ref{P has no leaf in B'} implies that every induced maximal even path in $H$ has order at most $4$.
\end{proof}

\begin{lemma}\label{lem:length<9}
	If $5 \leq m \leq n$ and $B$ is a basis of  $K_{1,m}\square K_{1,n}$, then the largest induced path in  $H := H(K_{1,m}\square K_{1,n}, B)$ has order at most $9$. Moreover, there is at most one induced path of order $9$ in $H$.
\end{lemma}

\begin{figure}[!htbp]
	\centering
	\includegraphics[width=.85\linewidth]{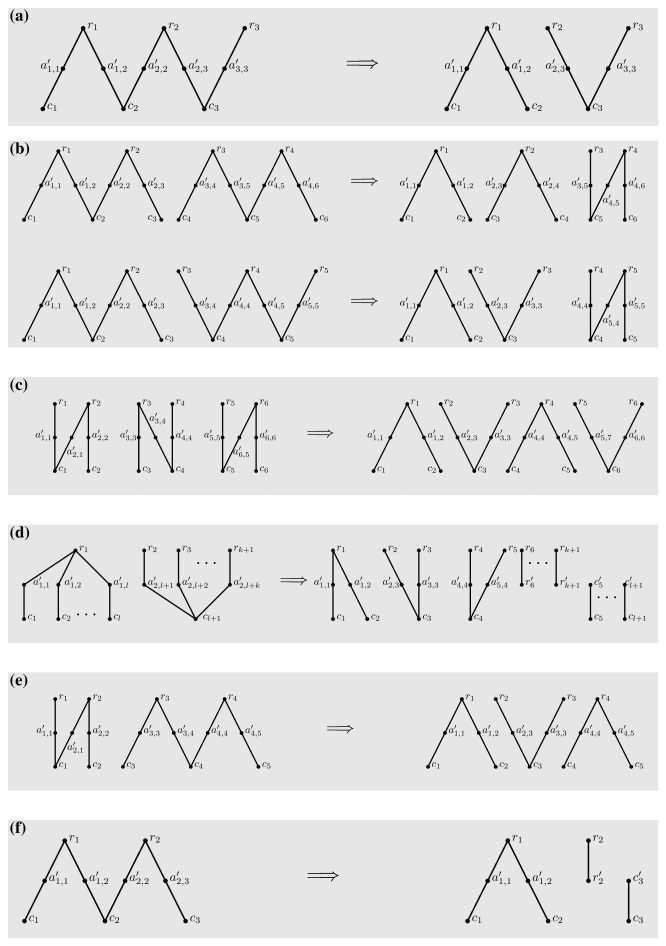}
	\caption{
		A set of six panels illustrating configurations and adjacency relations in the graph  $H(K_{1,m}\square K_{1,n}, B)$.
		(a) illustrates a path configuration of order 5 in a basis;
		(b) depicts the adjacency representation in a modified path structure;
		(c) shows an induced maximal path split into shorter resolving paths in a basis;
		(d) highlights the adjacency resolution of vertices in a modified $H(G, B)$;
		(e) represents resolving substructures in $H(G, B)$ to reduce basis size; and,
		(f) demonstrates adjacency in a specific subgraph of $H(G, B)$.
	}
	\label{fig:multiple}
\end{figure}

\begin{proof}
	By Corollary~\ref{no even maximal path>4}, it is sufficient to show that there is no path of order at least $11$ in $H$. Suppose, to the contrary, that $P$ is an induced maximal path in $H$ of order at least $11$. Lemma~\ref{P has no leaf in B'} implies that both leaves of $P$ are in $M \cup N$.
	
	By symmetry, we can assume that $c_1, b'_1, r_1, \ldots, c_3, b'_5, r_3$ are the first eleven vertices of $P$. With an argument similar to the proof of Lemma~\ref{P has no leaf in B'}, we can split this part of the path $P$ into two paths of order $5$ by removing $b'_3$, as shown in Figure~\ref{fig:multiple}(a). This gives us a resolving set for  $K_{1,m}\square K_{1,n}$ of cardinality $|B| - 1$, which is a contradiction. Thus, the largest induced path in $H$ has order at most $9$.

	Now, suppose there exist paths of order $9$ in $H$. Suppose, on the contrary, that there are two paths $P_1$ and $P_2$ of order $9$ in $H$. Consequently, $8$ vertices of $B'$ are shared between them.
	
	By Lemma~\ref{P has no leaf in B'} and the maximality assumption, all of these $8$ vertices have degree $2$. Thus, by symmetry, we can assume that these vertices are $a_{1,1}', a_{1,2}', a_{2,2}', a_{2,3}', a_{3,4}', a_{3,5}', a_{4,5}', a_{4,6}'$ or $a_{1,1}', a_{1,2}', a_{2,2}', a_{2,3}', a_{3,4}', a_{4,4}', a_{4,5}', a_{5,5}'$, as shown in Figure~\ref{fig:multiple}(b).
	
	Let $B_1 = (B \setminus \{a_{2,2}, a_{3,4}\}) \cup \{a_{2,4}\}$ or $B_1 = (B \setminus \{a_{2,2}, a_{3,4}, a_{3,5}\}) \cup \{a_{3,3}, a_{5,4}\}$, respectively. Clearly, $M \cup N$ is adjacency resolved by $B_1'$ in  $H(K_{1,m}\square K_{1,n}, B_1)$, and Lemma~\ref{dim(G)<|S|} implies that  $\dim(K_{1,m}\square K_{1,n}) \leq |B_1| = \dim(K_{1,m}\square K_{1,n}) - 1$. This contradiction completes the proof.
\end{proof}

\begin{lemma}\label{lem:at most two P7 in H}
	If $5 \leq m \leq n$ and $B$ is a basis of  $K_{1,m}\square K_{1,n}$, then there are at most two maximal induced paths of order $7$ in  $H := H(K_{1,m}\square K_{1,n}, B)$.
\end{lemma}

\begin{proof}
	Assume there are three maximal induced paths of order $7$ in $H$. By symmetry, let the $9$ vertices from $B'$ be $a_{1,1}', a_{2,1}', a_{2,2}', a_{3,3}', a_{3,4}', a_{4,4}', a_{5,5}', a_{6,5}', a_{6,6}'$, as shown in Figure~\ref{fig:multiple}(c).
	
	Let $B_1 = (B \setminus \{a_{2,1}, a_{2,2}, a_{3,4}, a_{5,5}, a_{6,5}\}) \cup \{a_{1,2}, a_{2,3}, a_{4,5}, a_{5,6}\}$. Hence, $M \cup N$ is adjacency resolved by $B_1'$ in  $H(K_{1,m}\square K_{1,n}, B_1)$, and by Lemma~\ref{dim(G)<|S|},  $\dim(K_{1,m}\square K_{1,n}) \leq |B_1| = \dim(K_{1,m}\square K_{1,n}) - 1$, which is a contradiction.
\end{proof}

\begin{lemma}\label{no path of order 3}
	If  $5 \leq m \leq n$ and $B$ is a basis of  $K_{1,m}\square K_{1,n}$, then there is no maximal induced path of order $3$ in  $H := H(K_{1,m}\square K_{1,n}, B)$.
\end{lemma}

\begin{proof}
	Suppose, to the contrary, that there exists a path $P$ of order $3$ in $H$. If $P$ has a leaf in $B'$, then $P$ has exactly two leaves in $B'$, and clearly $B'$ without one of them adjacency resolves $M \cup N$, which is a contradiction.
	
	Otherwise, $P$ has a leaf in $M$ and one in $N$. It is easy to see that the adjacency representations of these two leaves with respect to $B'$ are the same, which is impossible. Therefore, there is no maximal induced path of order $3$ in $H$.
\end{proof}

\begin{corollary}\label{cor:length<8}
	If $5 \leq m \leq n$ and $B$ is a basis of  $K_{1,m}\square K_{1,n}$, then all the maximal induced paths in  $H := H(K_{1,m}\square K_{1,n}, B)$ are of order $4$ or $5$, except possibly:\\[.1cm]
	- One path of order $9$, or\\
	- Two paths of order $7$, and possibly some single edges and isolated vertices.
\end{corollary}

\begin{lemma}\label{lem:Degree3}
	Let $5 \leq m \leq n$, and $B$ be a basis of  $K_{1,m}\square K_{1,n}$. If each of the parts $M$ and $N$ contains a vertex of degree at least $3$ in  $H := H(K_{1,m}\square K_{1,n}, B)$, then each of $M$ and $N$ contains exactly one vertex of degree $3$.
\end{lemma}

\begin{proof}
	Suppose that both $M$ and $N$ contain vertices of degree at least $3$ in $H$. First, we show that the maximum degree of each part is at most $3$. Assume to the contrary that $\deg_{H}(r_1) = l \geq 3$ for $r_1 \in M$ and $\deg_{H}(c_{l+1}) = k \geq 4$ for $c_{l+1} \in N$, with $l + k$ members of $B'$, as shown in Figure~\ref{fig:multiple}(d).
	
	Consider the set $B_1$, where the elements $a_{1,3}, \ldots, a_{1,l}, a_{2,l+1}, \ldots, a_{2,l+k}$ are removed from $B$ and replaced by $a_{2,3}$, $a_{3,3}$, $a_{4,4}$, $a_{5,4}$, $r_6$, $\ldots$, $r_{k+1}$, $c_5$, $\ldots$, $c_{l+1}$. 
	That is, $B_1$ is obtained by removing the elements  
	$\{a_{1,3}, \ldots, a_{1,l}, a_{2,l+1}, \ldots, a_{2,l+k}\} $ from $B$,  
	and adding the elements  
	$\{a_{2,3}, a_{3,3}, a_{4,4}, a_{5,4}, r_6, \ldots, r_{k+1}, c_5, \ldots, c_{l+1} \}$.
	Clearly, $|B_1| = |B| - 1$, and $M \cup N$ is adjacency resolved by $B_1'$ in  $H(K_{1,m}\square K_{1,n}, B_1)$. By Lemma~\ref{dim(G)<|S|}, this implies that  $\dim(K_{1,m}\square K_{1,n}) \leq |B_1| = \dim(K_{1,m}\square K_{1,n}) - 1$, contradicting the minimality of $B$.
	
	Therefore, each part contains at most one vertex of degree $3$. Figure~\ref{fig:multiple}(d) demonstrates that any more would contradict the minimality of $B$.
\end{proof}

\begin{lemma}\label{lem:path68}
	Let $5 \leq m \leq n$, and $B$ be a basis of  $ K_{1,m}\square K_{1,n} $. Then  $ H := H(K_{1,m}\square K_{1,n},B) $ cannot contain induced paths of order $7$ and $9$ simultaneously.
\end{lemma}

\begin{proof}
	If it does, then Figure~\ref{fig:multiple}(e) shows how we can find a basis of size less than $ |B| $.
\end{proof}

According to the previous lemmas, if $5 \leq m \leq n$ and $B$ is a basis of $ K_{1,m}\square K_{1,n} $, then $ H $  may consist of paths of order $1$, $2$, $4$, and $5$,  possibly one path of order $9$, and at most two paths of order $7$.  However, paths of order $7$ and $9$ cannot appear simultaneously.
Also, if both $ M $ and $ N $ contain a vertex of degree at least $3$, then each of them contains at most one  such vertex of degree $3$.
These results describe the possible configurations that may arise in $H$, as Lemmas~\ref{P has no leaf in B'}, \ref{lem:length<9}, \ref{lem:at most two P7 in H}, \ref{no path of order 3}, \ref{lem:Degree3}, \ref{lem:path68} and Corollaries~\ref{no even maximal path>4}, \ref{cor:length<8} apply to any basis $B$.  In what follows, however,  we shift our focus to find a basis $B$ with specific desired properties. Accordingly, the upcoming results are existential in nature.

\begin{theorem}\label{Thm:1or4}
	If $5 \leq m \leq n$, then there exists a basis $B$ of  $ K_{1,m}\square K_{1,n} $ such that  $ \Delta(H(K_{1,m}\square K_{1,n},B)) \leq 2 $ and  $H(K_{1,m}\square K_{1,n},B)$ consists only of  paths of order  $1$, $2$ and $5$.
\end{theorem}

\begin{proof}
	To see the first part, if there is a vertex $ v $ of degree $ r \geq 3 $ in  $ H := H(K_{1,m}\square K_{1,n}, B) $, then the corresponding spider graph can be replaced by a path of order $5$ with $ v $ as the middle vertex and $ r-2 $ edges. This also shows that there is a basis $ B $ of  $K_{1,m}\square K_{1,n}$ such that $ H $ may contain only some edges, paths of order $5$, at most two paths of order $7$, and at most one path of order $9$. However, a path of order $4$ can be changed into two edges, a path of order $9$ can be replaced by a path of order $5$ and two edges. Also, a path of order $7$ can be changed into a path of order $5$ and an edge as shown in Figure~\ref{fig:multiple}(f).
\end{proof}

\begin{lemma}\label{M and N has one end}
	Let $5 \leq m \leq n$. Then there exists a basis $B$ of  $K_{1,m}\square K_{1,n}$ such that  $H := H(K_{1,m}\square K_{1,n}, B)$ consists only of paths of order $1$, $2$, and $5$. Moreover, if both $M$ and $N$ intersect some path of order $2$ at exactly one vertex, then each of them intersects at most one such path.
\end{lemma}

\begin{proof}
	The first part of the statement follows directly from Theorem~\ref{Thm:1or4}. For the second part, suppose for contradiction that $M$ intersects two distinct single edges (i.e., paths of order $2$) and $N$ intersects one such edge. We will show that replacing these three single edges with a single path of order $5$ yields a smaller resolving set, contradicting the minimality of $B$.
	
	Assume the two single edges intersecting $M$ are $r_1b'_1$ and $r_2b'_2$, and the edge $c_1b'_3$ intersects $N$, where $r_1, r_2 \in M$, $c_1 \in N$, and $b_1, b_2, b_3 \in B$. Observe that the adjacency distance from each of $b'_1, b'_2, b'_3$ to any vertex in $M \cup N \setminus \{r_1, r_2, c_1\}$ in $H = H(G,B)$ is $2$. Therefore, the only vertices for which $b'_1, b'_2, b'_3$ may help distinguish adjacency are $r_1$, $r_2$, and $c_1$.

	Now, define a new set $B_1 := (B \setminus \{b_1, b_2, b_3\}) \cup \{a_{1,1}, a_{2,1}\}$. 
	In $H(G, B_1)$, the path $r_1a'_{1,1}c_1a'_{2,1}r_2$ has order $5$ and serves the same distinguishing function as the three original single edges for the vertices $r_1$, $r_2$, and $c_1$. This construction corresponds to replacing the three single edges with a path of order $5$, as proposed at the beginning of the proof.
	
	By Lemma~\ref{dim(G)<|S|}, we have $\dim(G) \leq |B_1|$. Since $|B_1| = |B| - 1$, this contradicts the minimality of $B$ as a basis for $G$.
\end{proof}

\begin{observation}\label{obs:1isolatexVertex}
	Let $5 \leq m \leq n$, and $B$ be a basis of $K_{1,m}\square K_{1,n}$. Then the graph $H := H(K_{1,m}\square K_{1,n}, B)$ contains at most one isolated vertex.
\end{observation}

\begin{proof}
	Observe that the adjacency representation of an isolated vertex in $H$ with respect to $B$ is $(2,2,\ldots,2)$. Since $B$ is an adjacency basis of $H$, every pair of distinct vertices must have different representations with respect to $B$. Therefore, there can be at most one isolated vertex in $H$.
\end{proof}

\begin{lemma}\label{lem:<=2isolatedEdges}
	If $5 \leq m \leq n$, then there exists a basis $B$ of $K_{1,m}\square K_{1,n}$ such that $ H := H(K_{1,m}\square K_{1,n}, B) $  consists only of at least one path of order $5$, at most one isolated vertex, and at most two single edges.
\end{lemma}

\begin{proof}
	By Theorem~\ref{Thm:1or4}, there is a basis for $G:=K_{1,m}\square K_{1,n}$ containing only paths of order $1$, $2$ and $5$. 
	Suppose, for contradiction, that for every such basis $B$ of $G$, the associated graph $H$ contains more than two single edges (paths of order $2$).
	Observation \ref{obs:1isolatexVertex} guarantees that $H=H(G, B)$ contains at most one isolated vertex.
	Thus, we may assume that $B$ is a basis such that $H$ consists only of paths of order $5$, at most one isolated vertex and at least three single edges.

	By the construction of $H$, every path of order $5$ intersects $M\cup N$ in three vertices, while a single edge or an isolated vertex intersects $M\cup N$ in exactly one vertex. Since $n \geq m \geq 5$, Lemma~\ref{M and N has one end} implies that at least one of the sets $M$ or $N$ 
	contains at least three vertices that cannot be covered, even collectively, by the isolated vertex and the single edges.
	Hence, 
	there must exist a path of order $5$ in $H$. As illustrated in Figure~\ref{fig:Axha}(a), we can replace this $P_5$ and the three single edges by a new configuration involving fewer vertices from $B$, contradicting the minimality of $B$.
	Therefore, there must exist a basis $B$ for which $H$ contains at most two single edges.
\end{proof}

\begin{figure}[!t]
	\centering
	\includegraphics[width=.85\linewidth]{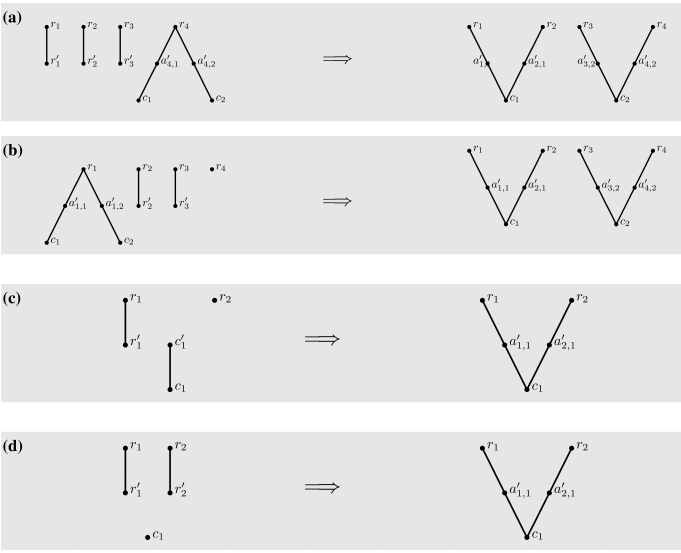}
	\caption{
		A set of four figures demonstrating adjacency relations and path resolutions in the graph $H(G, B)$ to analyze the basis and resolving properties:
		(a) depicts the adjacency resolving set applied to the graph $H(G, B)$;
		(b) illustrates resolving configurations for paths of order $4$;
		(c) shows adjacency in modified resolving structures; and,
		(d) demonstrates adjacency resolution with vertices of degree at least $3$.
	}
	\label{fig:Axha}
\end{figure}

\begin{lemma}\label{at most 2 single edges and isolated vertices}
	If $5 \leq m \leq n$, then there exists a basis $B$ of  $K_{1,m}\square K_{1,n}$ such that  $ H := H(K_{1,m}\square K_{1,n}, B) $ contains only of at least one path of order $5$, at most two single edges, and at most one isolated vertex. Moreover, the total number of  single edges and isolated vertices in  	$H := H(K_{1,m}\square K_{1,n}, B)$ is at most two.
	
\end{lemma}

\begin{proof}
	
	By Lemma~\ref{lem:<=2isolatedEdges}, there exists a basis $B$ of $G:=K_{1,m}\square K_{1,n}$ such that $H = H(G,B)$ consists of  at least one path of order $5$, at most two single edges, and at most one isolated vertex. Assume  
	that $H$ contains exactly two single edges and one isolated vertex. 
	If  all these three structures are contained in the same set  among $M$  or $N$, then  since $5 \leq m \leq n$, there  exists a path of  order $5$ in $H$ whose both leaves lie in the other set. Then, as shown in Figure~\ref{fig:Axha}(b), one can  construct a new basis $B'$ for which $H(G,B')$ contains neither a single edge nor an isolated vertex.  On the other hand, if the ends of the two single edges and  the isolated  vertex are  distributed across $M$ and $N$, then as in Figures~\ref{fig:Axha}(c) and \ref{fig:Axha}(d),  a new basis of $G$  can again be found such that $H$ contains no single edges or isolated vertices.
	
\end{proof}

\subsection{Case 2: $m\leq 4$ or $n\leq 4$}
To prove Proposition~\ref{prop:small_casesA}, we begin with a lower bound lemma that, in fact, holds on the entire interval
$\frac{n}{2} \leq m \leq n$.

\begin{lemma}\label{lem: n/2<m<n}
	If $\frac{n}{2} \leq m \leq n$, then $\dim(K_{1,m} \square K_{1,n}) \geq n$. 
\end{lemma}
\begin{proof}
	Assume, to the contrary, that $S$ is a resolving set for $K_{1,m} \square K_{1,n}$  with $|S| = n-1$. 
	We first show that $S$ intersects exactly $n-1$ columns, each with one vertex, and that there is one column with no vertex in $S$.  Suppose instead that $S$ intersects at most $n-2$ distinct columns. Then
	there exist two columns, say $j$ and $k$, such that neither contains any vertex from $S$. Therefore, $c_j$ and $c_k$ would have the same metric code with respect to $S$
	, contradicting the assumption that $S$ is a resolving set.
	Hence, $S$ must intersects exactly $n-1$ columns, each with exactly one vertex. This implies that $r_i \not\in S$ for all $1 \leq i \leq m$. 
	
	Next, we show that every row must contain at least one vertex from $S$. 
	Let $j'$ be the column that  contain no vertex from $S$. If there exists a row, $i'$,  that also contains no vertex from  $S$, then $r_{i'}$ and $c_{j'}$ would have identical metric codes with respect to $S$,  again contradicting the assumption that $S$ resolves the graph. Therefore, every row must contain at least one vertex from $S$.
	
	Finally, we claim that each row must contain at least two vertices from $S$.   Consider the $i$th row and let $a_{i,j} \in S$.  Since $a_{i,j} \in S$, $c_j$ cannot be in $S$ (as this column contains exactly one vertex from $S$).
	To distinguish $r_i$ from $c_j$, there must be at least one other vertex in $S$ either in the $i$th row or in the $j$th column. However, by the previous argument, each of the $n-1$ columns already contains exactly one vertex from $S$, namely $a_{i,j}$, so the required additional vertex must lie in the $i$th row. 
	Hence,  every row contains at least two vertices from $S$. This implies that, $|S| \geq 2m \geq n$,  contradicting the assumption that $|S| = n-1$. Therefore, $\dim(K_{1,m} \square K_{1,n}) \geq n$.
\end{proof}

\begin{proposition}\label{prop:small_casesA}
	Let ${n\over2}\leq m\leq n$, $2m-n<3$, and either $m\leq 4$ or $n\leq 4$. Then $\dim(K_{1,m}\square K_{1,n})=n$.
\end{proposition}

\begin{proof}
	The lower bound is clear by Lemma~\ref{lem: n/2<m<n}. To see the upper bound, it is enough to present a resolving set of size $n$. We will establish this by considering three possible cases.
	
	\textbf{Case (i): $2m-n=0$.}
	
	We claim that $S = \{a_{1,1}, a_{1,2}, a_{2,3}, a_{2,4}, \ldots, a_{m,n-1}, a_{m,n}\}$ is a resolving set of size $n$.
	
	Note that if $a_{i,j} \in S$, then elements of the $i$th row are resolved from other rows. Similarly, elements of the $j$th column are resolved from other columns. By the same argument, whenever $a_{i,j} \in S$, $r_i$ is resolved from $r_{i'}$ for all $i' \neq i$, and $c_j$ is resolved from $c_{j'}$ for all $j' \neq j$. Therefore, each row and each column of the graph has a member in $S$, ensuring that all elements of $A$, all $r_i$'s ($1 \leq i \leq m$), and all $c_j$'s ($1 \leq j \leq n$) are resolved from each other.
	
	Also, note that an element of $S$ is of odd distance with $r_i$'s and $c_j$'s but of even distance with elements of $A$. Therefore, the only candidates for two vertices to have the same metric code with respect to $S$ are $r_i$ and $c_j$ for some $i, j$. However, this is impossible due to the fact that each row has two members in $S$, while each column has only one. This completes the proof for the first case.
	
	\textbf{Case (ii): $2m-n=1$.}
	
	We define $S = \{a_{1,1}, a_{1,2}, a_{2,3}, a_{2,4}, \ldots, a_{m-2,n-2}, a_{m-1,n-1}, r_m\}$ as a resolving set of size $n$ for this case.
	
	All columns except the last one and all rows intersect $S$. Following a similar argument to the previous case, it is evident that all vertices except those in the last column and last row are resolved. Since $r_m$ is in $S$ and each column but the last one contains a member of $S$, all members of the $m$th row are resolved from each other. The only remaining vertices to be checked are those in the last column. It's clear that $c_n$ is the unique vertex with metric code $(3,3,\ldots, 3, 2)$. Additionally, $a_{i,n}$, where $1 \leq i \leq m-1$, is the only vertex with distance $4$ from every member of $S\setminus\{r_m\}$ and distance $2$ from $r_m$. Thus, $S$ serves as a resolving set in this case.
	
	\textbf{Case (iii): $2m-n=2$.}
	
	In this case, $S = \{a_{1,1}, a_{1,2}, a_{2,3}, a_{2,4}, \ldots, a_{m-1,n-1}, a_{m-1,n}\}$ serves as a resolving set of size $n$ for the graph.
	
	Following a similar argument to the first case, it suffices to check elements of the last row. The only vertex with all entries of its metric code equal to $3$ is $r_m$. Let $x$ be the unique vertex of $S$ in the $j$th column. Then, $a_{m, j}$ is the only vertex with distance $2$ from $x$ and $4$ from every element of $S\setminus\{x\}$. Thus, $S$ is a resolving set for $G$.
\end{proof}

With Proposition~\ref{prop:small_casesA} we have settled every configuration in which $\min\{m,n\}\le 4$ and $2m-n<3$.
The only unresolved cases with $m\leq 4$ or $n\leq 4$ therefore satisfy $2m-n\geq 3$.
To handle these instances we first show that, when $2m-n\geq 3$, no vertex of $M$ can belong to a basis of cardinality~$n$.

\begin{lemma}\label{lem:|S|=n, Sar Satr}
	Let ${n\over2}\leq m\leq n$, $2m-n\geq 3$. If $S$ is a basis of size $n$ for $K_{1,m}\square K_{1,n}$, then for any $1\leq i\leq m$, $r_i\not\in S$.
\end{lemma}
\begin{proof}
	Let $S$ be a basis of size $n$ for $G=K_{1,m}\square K_{1,n}$. Clearly, at most one of the columns can have no intersection with $S$. Now, if $S$ contains $r_i$ and $r_{i'}$ for distinct $i$ and $i'$, then $n=|S|\geq 2+n-1=n$ which is a contradiction. Thus, $S$ contains at most one $r_i$, for some $i$.
	
	We will show that this also cannot happen. Assume to the contrary that $r_{i_0}\in S$. Then each of the $n-1$ other elements of $S$ must intersect one distinct column, meaning there is exactly one column without any member of $S$. Without loss of generality, assume this is the last column.
	
	If the $j$th row does not intersect $S$, then $r_j$ and $c_n$ have the same metric code with respect to $S$. This implies that each row must intersect $S$. As $r_{i_0}$ is the only $r_i$ in $S$, this intersection for the $l$th row happens in $a_{l,k}$ for some $k\neq n$, and $l\neq i_0$.
	
	Now to resolve $r_l$ and $c_k$, there must be another member of $S$ on the $l$th row, i.e., every row except $i_0$ has at least two members in $S$. Thus, $n=|S|\geq 2(m-1) +1=2m-1\geq n+2$, which is a contradiction.
\end{proof}

\begin{remark}\label{rem:m=n}
	Let $m=n\geq 3$ and $S$ be a basis of size $n$ for  $K_{1,m}\square K_{1,n}$. By Lemma~\ref{lem:|S|=n, Sar Satr}, for any $1\leq i\leq m$, $r_i\not\in S$. As the graph is symmetric, one may assume there is no $c_j$ in $S$, repeat the argument, and conclude that not only for any $1\leq i\leq m$, $r_i\not\in S$ but also for any $1\leq j\leq n$, $c_j\not\in S$.
\end{remark}

Let $m\leq 4$ or $n\leq 4$. In Proposition~\ref{prop:small_casesA}, we considered the cases in which $2m-n<3$. Thus, to obtain the metric dimension of $K_{1,m}\square K_{1,n}$, it is sufficient to consider the cases in which $2m-n\geq 3$, which are, in fact, (i) $m=n=3$, (ii) $m=n=4$, and (iii) $m=4, n=5$.

\begin{proposition}\label{prop:small_casesB}
	Let ${n\over2}\leq m\leq n$, $2m-n\geq 3$, and either $m\leq 4$ or $n\leq 4$. Then $\dim(K_{1,m}\square K_{1,n}) =  n+\lfloor{2m-n\over 3}\rfloor = n+1$.
\end{proposition}
\begin{proof}
	First, we show that any resolving set for $G$ has at least $n+1$ elements. Then we provide a resolving set of size $n+1$.
	
	There are only three cases that we need to consider, due to the argument presented before the statement:
	(i) $m=n=3$,
	(ii) $m=n=4$, and
	(iii) $m=4, n=5$.
	
	Assume to the contrary that $S$ is a resolving set for $G$ and $|S|=n$. First of all, by Lemma~\ref{lem:|S|=n, Sar Satr}, for any $1\leq i\leq m$, $r_i\not\in S$. Also, if $n=m$, by Remark~\ref{rem:m=n}, for any $1\leq j\leq n$, $c_j\not\in S$ as well.
	
	We prove that this is also true in case (iii). Assume that $c_r, c_j\in S$ for $r\neq j$. Thus, we have $n-2=m-1$ remaining members in $S$. On the other hand, we have at most one row not intersecting $S$. This means that $m-1$ members of $S$ intersect $m-1$ rows each in exactly one $a_{l,k}$ for some $l, k$. Now assume that $a_{l,k}\in S$, where $k\not\in \{r,j\}$. Then to resolve $c_k$ and $r_l$, there must be another member of $S$ on the $k$th column. Therefore, $|S|\geq 2 + 2(m-1) = 2m = 2n-2 \geq n+1$, which is a contradiction as $n\geq 4$.
	
	Now, assume that there exists a unique column, say $j_0$, that does not intersect $S$. It means that there is a unique column, say $j_1$, with two members of $S$, and any other column except $j_0$ and $j_1$ has exactly one member of $S$. Note that we may have $j = j_1$.
	
	If the $l$th row does not intersect $S$, then $r_l$ and $c_{j_{0}}$ have the same metric code with respect to $S$. Hence, every row must intersect $S$. With a similar argument to the previous case, if $a_{l,k}\in S$ for some $1\leq l\leq m$ and $k\not\in \{j, j_1\}$, then row $l$ has to have another member of $S$ to resolve $r_l$ and $c_k$. Thus, each of these rows intersects $S$ in at least two members. Since there are at least $m-3$ rows of this type, $n=|S|\geq 3 + 2(m-3) + 1 = 2m-2 \geq n+1$, as $n\geq 5$. So far, we have proved that in all the cases (i)-(iii), $S \subseteq A$. As discussed earlier, if $a_{i,j}\in S$, then row $i$ has another member in $S$ to resolve $r_i$ and $c_j$. On the other hand, at most one row has no member in $S$. Therefore, $n=|S|\geq 2(m-1)=2(n-2)$, which is a contradiction as $n\geq 3$. This completes the proof for the lower bound.
	
	To see the upper bound, set $S=\{a_{1,1}, a_{1,2}, a_{2,3}, a_{3,3}\}$ if $m=n=3$; $S=\{a_{1,1}, a_{1,2}, a_{2,3}, a_{3,3}, r_4\}$ if $m=n=4$; and $S=\{a_{1,1}, a_{1,2}, a_{2,3}, a_{3,3}, a_{4,4}, a_{4,5}\}$ if $m=4, n=5$.
\end{proof}

\subsection{Proof of Theorem~\ref{n/2<m}}
We are now ready to establish the main result for the regime $\frac{n}{2} \leq m\leq n$.
\begin{proof}
	Let $G=K_{1,m}\square K_{1,n}$.
	We will show that $G$ has a basis  of cardinality $n+\lfloor{2m-n\over 3}\rfloor$.

	If $m \leq 4$ or $n \leq 4$, the claim follows from Propositions~\ref{prop:small_casesA} and \ref{prop:small_casesB}. Otherwise, if $5 \leq m \leq n$, then
	by Lemma~\ref{at most 2 single edges and isolated vertices}, there exists a basis $B$ of $G$  such that $ H := H(G,B) $ consists only of at least one path of order $5$, at most two single edges, and at most one isolated vertex. Moreover, the total number of single edges and isolated vertices in $H = H(G,B)$ is at most two. It is also straightforward that a path of order $5$ in $H$ intersects $M \cup N$ in three vertices, while a single edge or an isolated vertex intersects $M \cup N$ in exactly one vertex.
	
	Let $m+n=3k+r$, for some integers  $k$ and $r$ with $0\leq r\leq 2$ and $r<k$. We consider the following cases:
	
	\textbf{Case 1:} $r=0$.

	Here, $m+n$ is divisible by three, thus the total number of single edges and isolated vertices must also be divisible by three. However, by Lemma~\ref{at most 2 single edges and isolated vertices}, this number is at most two, which forces it to be zero. Hence, $H$ contains no single edges or an isolated vertices.
	Therefore, $H$ is disjoint union of some paths of order $5$. Let $s$ of these paths  contain a vertex of degree $2$ in $M$, and let $t$ of them  contain such a vertex in $N$.  Then:
	$$
	\left\{
	\begin{array}{ll}
		s + 2t = n \\
		2s + t = m
	\end{array}
	\right.
	$$
	Solving this system gives $s = \frac{2m-n}{3}$. Since $m+n$ is divisible by $3$, the remainder of $2m-n$ divided by $3$ is $0$, and $|B| = 2s + (n - s) = n + s = n + \lfloor \frac{2m-n}{3} \rfloor$.
	
	\textbf{Case 2:} $r=1$.

	In this case, in addition to some paths of order $5$, $H$ may contain at most one single edge or one isolated vertex.
	
	If the isolated vertex  or an endpoint of the single edge lies in $N$,  set $n' = n-1$, and  apply the same reasoning as in Case $1$. Then, $|B| = n-1 + \frac{2m-(n-1)}{3}$. Note that $r=1$ implies that the remainder of $2m-n$ divided by $3$ is $2$, and $|B| = n-1 + \lfloor \frac{2m-n}{3} \rfloor + 1 = n + \lfloor \frac{2m-n}{3} \rfloor$.

	If  instead the isolated vertex  or an endpoint of the single edge is in $M$, set $m' = m-1$.  Then,   $|B| = n + \frac{2(m-1)-n}{3} = n + \lfloor \frac{2m-n}{3} \rfloor$.
	
	\textbf{Case 3:} $r=2$.

	Here, $H$  consists of some paths of order $5$, and either one single edge and one isolated vertex, or  two single edges. Note that in this case, the remainder of $2m-n$ divided by $3$ is $1$. 
	
	Depending on the placement of these isolated elements,we consider the following three subcases:
	
	- If one endpoint of a single edge is in  one of the sets $M$ or $N$ and the isolated vertex  or an endpoint of the other single edge is in the other set, then put $n' = n-1$ and $m' = n-1$.  Substituting these values, 
	$$ |B| = n-1 + \frac{2(m-1)-(n-1)}{3} + 1 = n + \lfloor \frac{2m-n}{3} \rfloor. $$
	
	- If the endpoint of the single edge and the isolated vertex  (or an endpoint of the other single edge)  are both in $M$, set $m' = m-2$. Then,
	$$ |B| = n + \frac{2(m-2)-n}{3} + 1 = n + \lfloor \frac{2m-n}{3} \rfloor. $$
	
	- Similarly, if the endpoint of the single edge and the isolated vertex  (or an endpoint of the other single edge)  are both in $N$,  set $n' = n-2$.  Then,
	$$ |B| = n-2 + \frac{2m-(n-2)}{3} + 1 = n + \lfloor \frac{2m-n}{3} \rfloor. $$
\end{proof}

\section{Practical Applications}
The Cartesian product of two star graphs, $K_{1,m}\square K_{1,n}$, arises naturally in structured sensing layouts that combine a single master station with two orthogonal sets of relay lines.  The central vertex $a_{0,0}$ corresponds to a master controller or gateway.  Horizontal relays $r_1,\dots ,r_m$ and vertical relays $c_1,\dots ,c_n$ act as intermediate collectors, while the peripheral vertices $a_{i,j}$ designate the physical positions where transducers are installed.  Each $a_{i,j}$ is linked to exactly two relays, its row-relay $r_i$ and its column-relay $c_j$, so that messages always have two independent paths to the master.  This simple-yet-redundant wiring pattern is easy to assemble in practice and easy to reason about mathematically; 
our results then determine both how few landmarks are needed and where to place them to ensure that every node remains distinguishable at all times.

Consider first a precision-agriculture deployment.  A greenhouse often contains parallel irrigation pipes or cable trays running east-west and another set running north-south.  Wireless or wired gateways are mounted at the ends of each east-west pipe (forming ${r_i}$'s) and at the ends of each north-south pipe (forming ${c_j}$'s), all reporting to a central automation computer $a_{0,0}$.  Moisture, nutrient, and temperature sensors are installed at the pipe intersections, these are the $a_{i,j}$ vertices.  From an economic standpoint it is infeasible to treat all of these sensors as high-end nodes with long-range radios or large batteries.  Instead, the farmer designates a basis, the minimal resolving set found by our $\mathcal O(m+n)$ algorithm, as landmark nodes equipped with stronger radios (or wired back-haul).  By design, every sensor has a unique hop-count vector to those landmarks.  Suppose a low-power sensor located at $a_{i,j}$ can report that its message reaches landmark $L_3$ in two hops and landmark $L_7$ in three hops.  If no other vertex in the grid shares this distance vector, the control computer deduces unambiguously that $a_{i,j}$ is the sender.  The hop-count signature thus acts as the address, avoiding the need for long identifiers or high-power radios at ordinary nodes.  Reporting a compact hop-count vector instead of explicit row-column coordinates mirrors common practice in low-power ZigBee, LoRaWAN-Class A, and TinyOS networks, where minimising payload size directly translates into lower airtime and longer battery life.  Choosing the landmark indices from our constructive proofs (for example, the $n-1$ vertices of Theorem~\ref{2<m<n/2} when $1<m<\tfrac n2$, or the $n+\lfloor \frac{2m-n}{3}\rfloor$ vertices of Theorem~\ref{n/2<m} otherwise) guarantees that no extra hardware is purchased and no intersection is left ambiguous.

A second illustration is real-time safety monitoring inside a warehouse.  The building's main fire-panel at the center plays the role of $a_{0,0}$.  Addressable loops radiate horizontally to floor-level junction boxes $r_i$ mounted along one long wall, and vertically to junction boxes $c_j$ along the adjacent wall.  Smoke, heat, or gas detectors placed at ceiling-grid intersections constitute the $a_{i,j}$ vertices.  Outfitting every detector with a full communication board would be costly and power-hungry, yet the fire code requires that the control room identify the exact detector in alarm within seconds.  Deploying the landmark set prescribed by our metric basis means that only those few detectors need addressable isolator modules; every other detector can be a simpler, cheaper unit that reports its hop distances to the isolators.  

Sensor networks also face operational challenges including limited range, communication noise, and partial failures.  Although hop-count or signal-strength readings are noisy, the distance vector (metric representation) assigned by a resolving set is highly tolerant to such uncertainty.  Any two vertices differ in at least one coordinate of their ideal code; and, in this grid, usually by two or more hops; so a small $\pm 1$-hop error rarely collapses two codes into an identical pattern.  The landmark codes thus act as short error-correcting signatures: the sink can choose the vertex whose ideal code is nearest (in Hamming distance) to the noisy vector it receives, and mis-identification remains improbable unless several coordinates are perturbed simultaneously. The basis therefore provides a provable lower bound on the number of high-quality landmarks required; engineers may add extra landmarks only if empirical noise exceeds this theoretical tolerance.  While resolving sets do not explicitly enforce proximity between landmarks and all nodes, they provide a robust baseline: each node remains distinguishable through its distance vector even when some measurements are uncertain or incomplete.  Reducing the number of landmarks lowers overall message volume, eases bandwidth demand, conserves battery power, and shortens maintenance checklists, because only a handful of critical landmarks need regular service.

Therefore, the bases developed in Section~\ref{Sec:Results} are more than combinatorial curiosities: they offer exact prescriptions for placing the fewest well-instrumented nodes needed to keep large, grid-structured (wired or wireless) sensor arrays fully and unambiguously observable.

\section{Discussion}

The metric dimension of a graph is a critical parameter with various applications in network navigation, chemistry, and robotics, among other fields,   some of which are briefly  discussed in the Introduction. In this paper, we have determined the exact metric dimension of the Cartesian product of star graphs, specifically $ K_{1,m} \square K_{1,n}$, for all values of $m$ and $n$.

Our primary contributions are encapsulated in Theorems \ref{2<m<n/2} and  \ref{n/2<m}, which establish the metric dimension for the cases of $2 \leq m < \frac{n}{2}$ and $\frac{n}{2} \leq m \leq n$, respectively. Specifically, we showed that when $2 \leq m < \frac{n}{2}$, the metric dimension of $K_{1,m} \square K_{1,n}$ is $n-1$, while for $\frac{n}{2} \leq m \leq n$, it is $n + \lfloor \frac{2m-n}{3} \rfloor$. Moreover, to ensure comprehensive coverage of all configurations, we considered the special case of $m=1$ in Theorem \ref{thm:SmallCases}. This result completes the analysis for all values of $m$ and $n$.

Our metric dimension results yield theoretical lower bounds and constructive blueprints that can inform, initialise, or benchmark graph based localisation pipelines \cite{zhang2023factorgraph} and learning assisted sensor placement schemes \cite{huang2025graphpool}.

To visualize these findings, we plotted the linear extension of the metric dimension of $K_{1,m} \square K_{1,n}$ against $m$ for a fixed $n$ (as shown in Figure~\ref{fig:2D}). This plot illustrates how the metric dimension changes with varying values of $m$. Specifically, we fixed  $n=14$, computed the metric dimension for $2 \leq m <\frac{n}{2}$, and then for $\frac{n}{2} \leq m \leq n$. 

\begin{figure}[h]
	\centering
	\includegraphics[width=.5\linewidth]{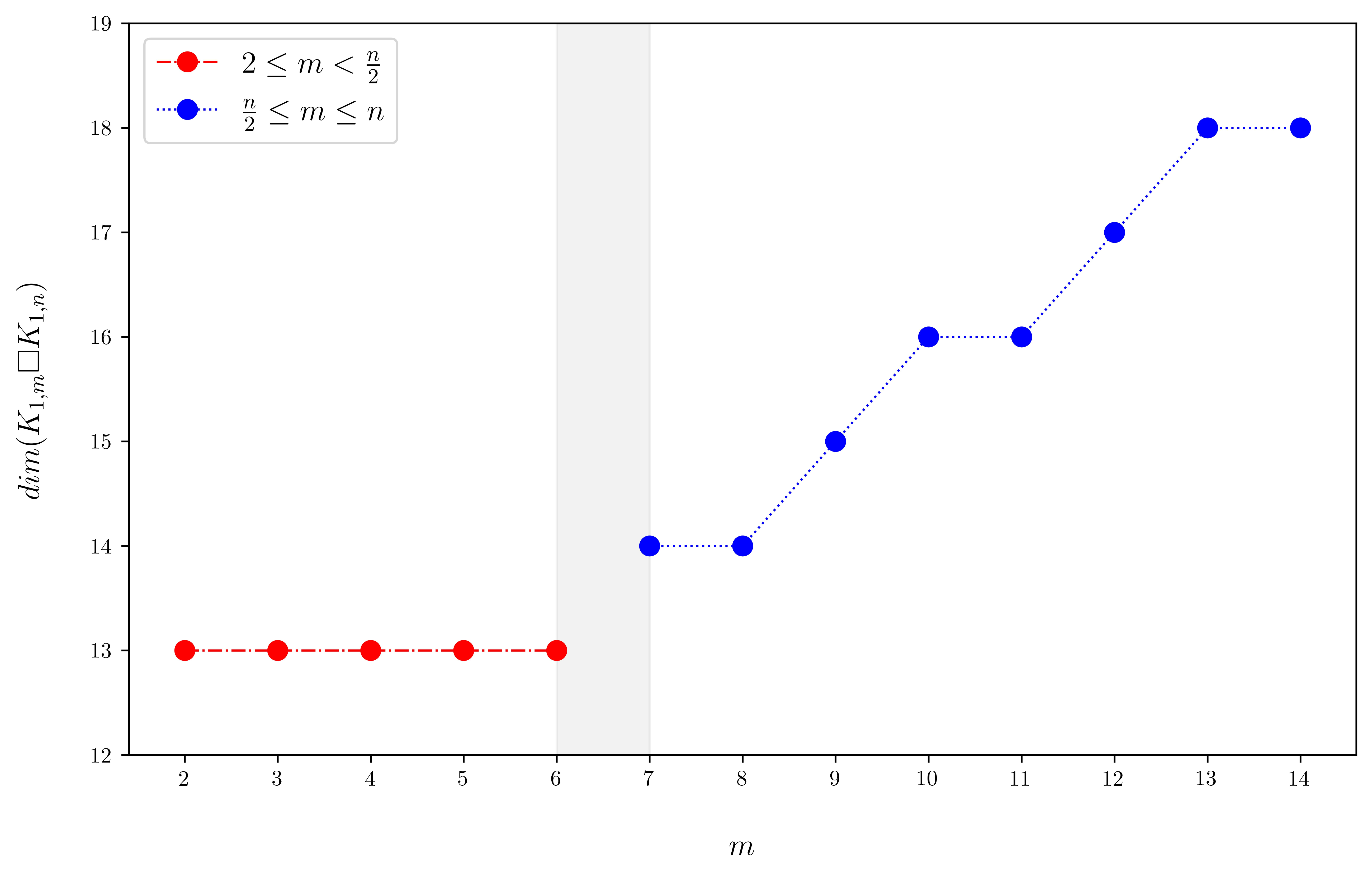} 
	\caption{Linear extension plot of the metric dimension $\dim(K_{1,m} \square K_{1,n})$ against $m$ for a fixed value of $n=14$. The plot shows the change in metric dimension as $m$ varies within the ranges $2 \leq m <\frac{n}{2}$ and $\frac{n}{2} \leq m \leq n$.}
	\label{fig:2D}
\end{figure}

Furthermore, in Figure~\ref{fig:3D}, we presented a heatmap that visualizes the metric dimension $\dim(K_{1,m} \square K_{1,n})$ as a function of both $m$ and $n$.  This visualization provides a comprehensive view of how the metric dimension varies with different values of $n$ and $m$. To aid in understanding, we depicted $n$ on the $x$-axis and $m$ on the $y$-axis, with colors indicating the metric dimension for each pair of $n$ and $m$, satisfying $2\leq m\leq n\leq n_0$. Here, $n_0$ is an arbitrary fixed integer upper bound for parameters $m$ and $n$ considered to plot the heatmap. Note that, for simplicity, we left the region where $m > n$ white, as we only consider cases where $m \leq n$.

\begin{figure}[h]
	\centering
	\includegraphics[width=.5\linewidth]{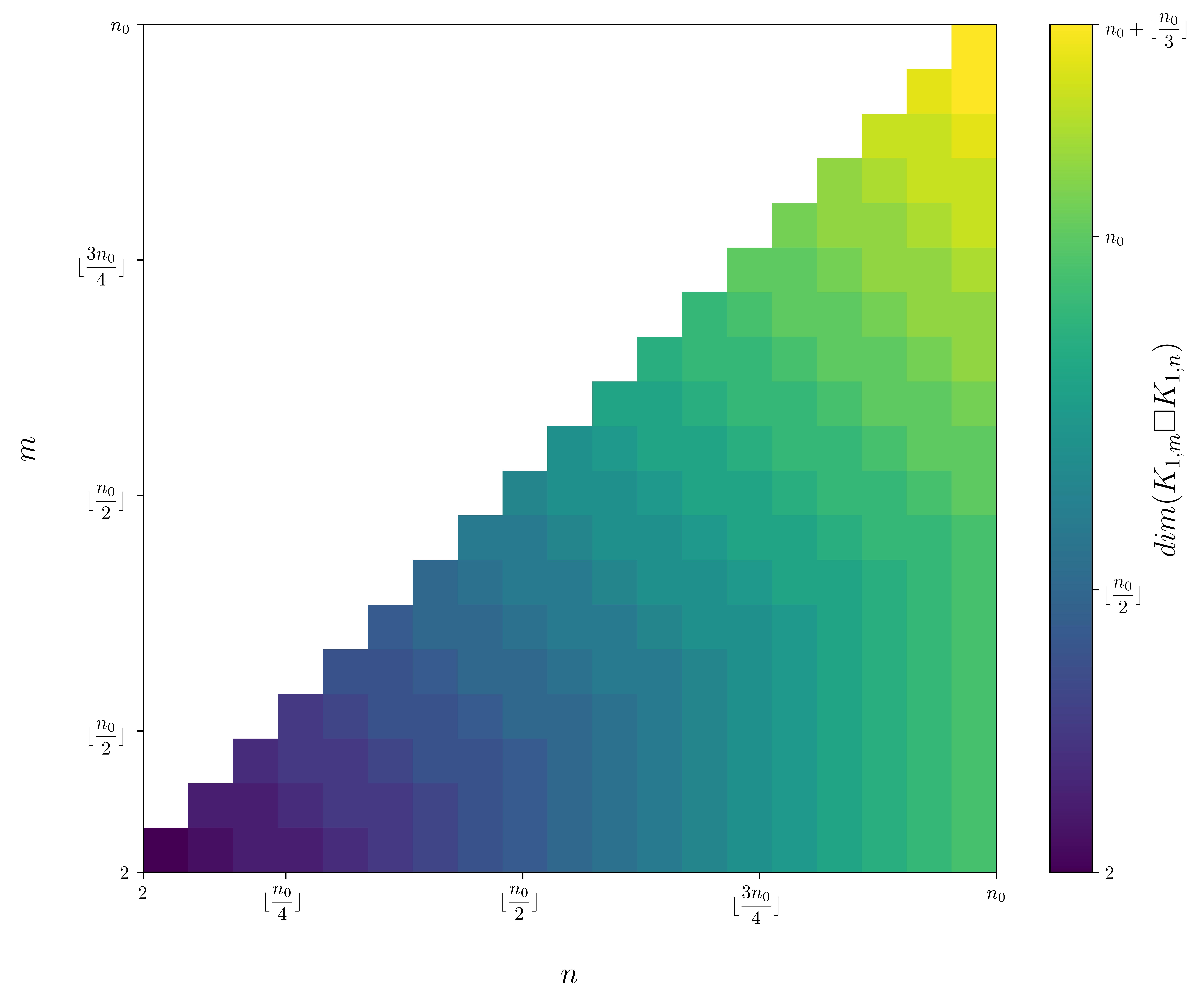}
	\caption{ Heatmap of $\dim(K_{1,m} \square K_{1,n})$ as a function of $m$ and $n$. The $x$-axis represents $n$, the $y$-axis represents $m$, and the colors indicate the metric dimension for each pair of $m$ and $n$.}
	\label{fig:3D}
\end{figure}

These visualizations simplify the comprehension of our theoretical findings and provide an intuitive understanding of the relationship between the metric dimension and the size of input star graphs. Additionally, we hope that our proof techniques and approaches will aid researchers in studying the metric dimensions of other graph products, thereby opening new avenues for future research in graph theory and its applications.

Beyond Cartesian products of star graphs, the methods developed in this paper could potentially be adapted to study the metric dimension of Cartesian products involving other families of graphs, and even more general graph operations such as the strong product ($K_{1,m}\boxtimes K_{1,n}$) and the direct product ($K_{1,m}\times K_{1,n}$). Moreover, our closed-form solutions and constructive $O(m+n)$ algorithm for $K_{1,m}\square K_{1,n}$ yield new benchmark instances for testing and calibrating distance-based invariants (e.g.\ metric dimension, locating-domination~\cite{locating-domination}, partition dimension~\cite{partitiondimension}). From an application standpoint, the same ``hub and spoke grid'' abstraction may capture layered network architectures, such as mitochondrial signaling motifs, city-to-airport transport systems, or data-center fabrics, where metric dimension quantifies the minimum monitor budget (landmarks) needed for unambiguous localization by hop-count alone. In chemical graph theory, Cartesian products of stars model repeating branch-and-chain motifs, and metric dimension corresponds to the resolving code length, i.e.\ the smallest set of spectroscopic tags needed to distinguish all atomic positions.

Finally, by interpreting a resolving set as a cover of the \emph{pair-distinction hypergraph}, whose vertices are the unordered pairs $\{u,v\}$ and where each hyperedge $S_w=\{\{u,v\}:d(u,w)\neq d(v,w)\}$ records which pairs are distinguished by a landmark $w$, one observes deep parallels with clique covering invariants~\cite{pullman}, including clique covers~\cite{triangle}, sigma clique covers~\cite{scc, scp}, and local clique covers~\cite{lcc}.  Techniques developed in that literature may thus migrate to metric dimension problems, while our explicit constructions suggest extremal patterns for these covering variants. 
Also, as our proofs rely on decomposing $K_{1,m}\square K_{1,n}$ into ``rows'' and ``columns'' centered at the high-degree vertices, we expect these techniques to extend to products $K_{1,m}\square G$ when $G$ is a tree or cactus, and even to strong  or direct  products that retain sufficient separability.

\section*{Declarations}

\section*{Competing interest}
The authors declare that they have no known competing financial interests or personal relationships that could have appeared to influence the work reported in this paper.

\section*{Author contributions}
All authors contributed equally to the study conception, design, and writing of the manuscript.

 \bibliographystyle{unsrt2authabbrvpp3}
\bibliography{cas-refs}

@article{adar2018metric, 
	title={The Metric Dimension of Two-Dimensional Extended Meshes}, volume={23}, DOI={10.14232/actacyb.23.3.2018.2}, abstractNote={We consider two-dimensional grids with diagonals, also called extended meshes or meshes. Such a graph consists of vertices of the form (i, j) for 1 ≤ i ≤ m and 1 ≤ j ≤ n, for given m, n ≥ 2. Two vertices are defined to be adjacent if the `∞ distance between their vectors is equal to 1. A landmark set is a subset of vertices L ⊆ V , such that for any distinct pair of vertices u, v ∈ V , there exists a vertex of L with different distances to u and v. We analyze the metric dimension and show how to obtain a landmark set of minimum cardinality.}, number={3}, journal={Acta Cybernetica}, author={Adar, Ron and Epstein, Leah}, year={2018}, month={Jan.}, pages={761-772} }

@ARTICLE{net2,
	author={Beerliova, Zuzana and Eberhard, Felix and Erlebach, Thomas and Hall, Alexander and Hoffmann, Michael and Mihal'ak, Mat and Ram, L. Shankar},
	journal={IEEE J Sel Areas Commun}, 
	title={Network Discovery and Verification}, 
	year={2006},
	volume={24},
	number={12},
	pages={2168-2181},
	doi={10.1109/JSAC.2006.884015}}

@article{2-trees,
	author = {Behtoei, Ali and Davoodi, Akbar and Jannesari, Mohsen and Omoomi, Behnaz},
	title = {A characterization of some graphs with metric dimension two},
	journal = {Discrete Math Algorithms Appl},
	volume = {09},
	number = {02},
	pages = {1750027},
	year = {2017},
	doi = {10.1142/S1793830917500276}
}

@article{bailey2013resolving,
	title = {Resolving sets for Johnson and Kneser graphs},
	journal = {Eur. J. Comb.},
	volume = {34},
	number = {4},
	pages = {736-751},
	year = {2013},
	issn = {0195-6698},
	doi = {10.1016/j.ejc.2012.10.008},
	author = {Robert F. Bailey and José Cáceres and Delia Garijo and Antonio González and Alberto Márquez and Karen Meagher and María Luz Puertas}
}

@article{code,
	title={Error-Correcting Codes from-Resolving Sets},
	author={Bailey, Robert F and Yero, Ismael G},
	journal={Discuss. Math. Graph Theory},
	volume={39},
	number={2},
	pages={341--355},
	year={2019},
	doi={10.7151/dmgt.2087}
}

@article{Ollerman,
	title = {Resolvability in graphs and the metric dimension of a graph},
	journal = {Discret. Appl. Math.},
	volume = {105},
	number = {1},
	pages = {99-113},
	year = {2000},
	issn = {0166-218X},
	doi = {10.1016/S0166-218X(00)00198-0},
	author = {Gary Chartrand and Linda Eroh and Mark A. Johnson and Ortrud R. Oellermann}
}

@article{ChemChart,
	title = {Resolvability and the upper dimension of graphs},
	journal = {Comput. Math. Appl.},
	volume = {39},
	number = {12},
	pages = {19-28},
	year = {2000},
	issn = {0898-1221},
	doi = {10.1016/S0898-1221(00)00126-7},
	author = {G. Chartrand and C. Poisson and P. Zhang}
}

@article{CartesianProduct,
	author = { C\'{a}ceres, Jos\'{e} and  Hernando, Carmen and  Mora, Merc\`{e} and  Pelayo, Ignacio M. and  Puertas, Mar\'{\i}a L. and  Seara, Carlos and  Wood, David R.},
	title = {On the Metric Dimension of Cartesian Products of Graphs},
	journal = {SIAM J Discret Math},
	volume = {21},
	number = {2},
	pages = {423-441},
	year = {2007},
	doi = {10.1137/050641867}
}

@InProceedings{Diaz,
	author="D{\'i}az, Josep
	and Pottonen, Olli
	and Serna, Maria
	and van Leeuwen, Erik Jan",
	editor="Epstein, Leah
	and Ferragina, Paolo",
	title="On the Complexity of Metric Dimension",
	booktitle="Algorithms -- ESA 2012",
	year="2012",
	publisher="Springer Berlin Heidelberg",
	address="Berlin, Heidelberg",
	pages="419--430"
}

@article{Levin,
	author = {Epstein, Leah and Levin, Asaf and Woeginger, Gerhard J.},
	date = {2015/08/01},
	doi = {10.1007/s00453-014-9896-2},
	id = {Epstein2015},
	isbn = {1432-0541},
	journal = {Algorithmica},
	number = {4},
	pages = {1130--1171},
	title = {The (Weighted) Metric Dimension of Graphs: Hard and Easy Cases},
	volume = {72},
	year = {2015}
}

@article{fijavvz2004rigidity,
	title = {Rigidity and separation indices of Paley graphs},
	journal = {Discrete Math.},
	volume = {289},
	number = {1},
	pages = {157-161},
	year = {2004},
	issn = {0012-365X},
	doi = {10.1016/j.disc.2004.09.004},
	author = {Gašper Fijavž and Bojan Mohar}
}

@article{fehr2006metric,
	title = {The metric dimension of Cayley digraphs},
	journal = {Discrete Math.},
	volume = {306},
	number = {1},
	pages = {31-41},
	year = {2006},
	issn = {0012-365X},
	doi = {10.1016/j.disc.2005.09.015},
	author = {Melodie Fehr and Shonda Gosselin and Ortrud R. Oellermann}
}

@article{Harary,
	title={On the metric dimension of a graph},
	author={Harary, Frank and Melter, Robert A},
	journal={Ars combin},
	volume={2},
	number={191-195},
	pages={1},
	year={1976}
}

@article{caceres2005metric,
	title = {On the metric dimension of some families of graphs},
	journal = {Electron. Notes Discrete Math.},
	volume = {22},
	pages = {129-133},
	year = {2005},
	issn = {1571-0653},
	doi = {10.1016/j.endm.2005.06.023},
	author = {Carmen Hernando and Mercè Mora and Ignacio M. Pelayo and Carlos Seara and José Cáceres and Mari L. Puertas}
}

@article{doubly,
	title = {Graphs with doubly resolving number 2},
	journal = {Discret. Appl. Math.},
	volume = {339},
	pages = {178-183},
	year = {2023},
	issn = {0166-218X},
	doi = {10.1016/j.dam.2023.06.017},
	author = {Mohsen Jannesari}
}

@Article{imran2018metric,
	AUTHOR = {Imran, Shahid and Siddiqui, Muhammad Kamran and Imran, Muhammad and Hussain, Muhammad},
	TITLE = {On Metric Dimensions of Symmetric Graphs Obtained by Rooted Product},
	JOURNAL = {Mathematics},
	VOLUME = {6},
	YEAR = {2018},
	NUMBER = {10},
	ARTICLE-NUMBER = {191},
	ISSN = {2227-7390},
	DOI = {10.3390/math6100191}
}

@article{iswadi2008metric,
	title={The metric dimension of graph with pendant edges},
	author={Iswadi, Hazrul and Baskoro, Edy Tri and Simanjuntak, Rinovia and Salman, ANM},
	journal={JCMCC},
	volume={65},
	pages={139--145},
	year={2008},
	publisher={Charles Babbage Research Center}
}

@article{n-3,
	title={Characterization of n-vertex graphs with metric dimension n-3},
	author={Jannesari, Mohsen and Omoomi, Behnaz},
	journal={Math. Bohem.},
	volume={139},
	pages={1--23},
	year={2014},
	doi={10.21136/MB.2014.143632}
}

@article{random,
	title = {On randomly k-dimensional graphs},
	journal = {Appl. Math. Lett.},
	volume = {24},
	number = {10},
	pages = {1625-1629},
	year = {2011},
	issn = {0893-9659},
	doi = {10.1016/j.aml.2011.03.024},
	author = {Mohsen Jannesari and Behnaz Omoomi}
}

@article{lexico,
	title = {The metric dimension of the lexicographic product of graphs},
	journal = {Discrete Math.},
	volume = {312},
	number = {22},
	pages = {3349-3356},
	year = {2012},
	issn = {0012-365X},
	doi = {10.1016/j.disc.2012.07.025},
	author = {Mohsen Jannesari and Behnaz Omoomi}
}

@article{javaid2008families,
	title={Families of regular graphs with constant metric dimension},
	author={Javaid, Imran and Rahim, M Tariq and Ali, Kashif},
	journal={Util. Math.},
	volume={75},
	number={1},
	pages={21--33},
	year={2008},
	publisher={Winnipeg: University of Manitoba, Department of Computer Science, 1972-}
}

@article{chem,
	author = {Mark Johnson},
	title = {Structure-activity maps for visualizing the graph variables arising in drug design},
	journal = {J. Biopharm. Stat.},
	volume = {3},
	number = {2},
	pages = {203--236},
	year = {1993},
	publisher = {Taylor \& Francis},
	doi = {10.1080/10543409308835060}
}

@article{landmarks,
	title = {Landmarks in graphs},
	journal = {Discret. Appl. Math.},
	volume = {70},
	number = {3},
	pages = {217-229},
	year = {1996},
	issn = {0166-218X},
	doi = {10.1016/0166-218X(95)00106-2},
	author = {Samir Khuller and Balaji Raghavachari and Azriel Rosenfeld}
}

@article{digital,
	title = {Metric bases in digital geometry},
	journal = {Comput.Gr.Image Process.},
	volume = {25},
	number = {1},
	pages = {113-121},
	year = {1984},
	issn = {0734-189X},
	doi = {10.1016/0734-189X(84)90051-3},
	author = {Robert A Melter and Ioan Tomescu}
}

@article{FatTree,
	author = {Prabhu, S. and Manimozhi, V. and Davoodi, Akbar and Guirao, Juan Luis Garc{\'\i}a},
	date = {2024/07/01},
	doi = {10.1007/s11227-024-06053-5},
	id = {Prabhu2024},
	isbn = {1573-0484},
	journal = {J. Supercomput.},
	number = {11},
	pages = {15783--15798},
	title = {Fault-tolerant basis of generalized fat trees and perfect binary tree derived architectures},
	volume = {80},
	year = {2024}}

@article{coin,
	title={On metric generators of graphs},
	author={Seb{\H{o}}, Andr{\'a}s and Tannier, Eric},
	journal={Math. Oper. Res.},
	volume={29},
	number={2},
	pages={383--393},
	year={2004},
	doi={10.1287/moor.1030.0070},
	publisher={INFORMS}
}

@article{ahmad2013metric,
	author = {Shabbir Ahmad, Muhammad Anwar Chaudhry, Imran Javaid and Muhammad Salman},
	title = {On the metric dimension of generalized Petersen graphs},
	journal = {Quaest. Math.},
	volume = {36},
	number = {3},
	pages = {421--435},
	year = {2013},
	publisher = {Taylor \& Francis},
	doi = {10.2989/16073606.2013.779957},
	
}

@article{saputro2013metric,
	title = {The metric dimension of the lexicographic product of graphs},
	journal = {Discrete Math.},
	volume = {313},
	number = {9},
	pages = {1045-1051},
	year = {2013},
	issn = {0012-365X},
	doi = {10.1016/j.disc.2013.01.021},
	author = {S.W. Saputro and R. Simanjuntak and S. Uttunggadewa and H. Assiyatun and E.T. Baskoro and A.N.M. Salman and M. Bača}
}

@article{rodriguez2015metric,
	ISSN = {15842851, 18434401},
	author = {Juan A. Rodríguez-Velázquez and Dorota Kuziak and Ismael G. Yero and José M. Sigarreta},
	journal = {Carpathian J. Math.},
	number = {2},
	pages = {261--268},
	publisher = {Sinus Association},
	title = {The metric dimension of strong product graphs},
	urldate = {2024-09-20},
	volume = {31},
	year = {2015}
}

@article{kuziak2017resolvability,
	author = {Kuziak, Dorota and Peterin, Iztok and Yero, Ismael G.},
	date = {2017/02/01},
	doi = {10.1007/s00025-016-0563-6},
	id = {Kuziak2017},
	isbn = {1420-9012},
	journal = {Results Math.},
	number = {1},
	pages = {509--526},
	title = {Resolvability and Strong Resolvability in the Direct Product of Graphs},
	volume = {71},
	year = {2017}}

@article{VillarceauGrids,
	title={Metric Dimension of Villarceau Grids},
	author={Prabhu, S and Jeba, D and Manuel, Paul and Davoodi, Akbar},
	journal={arXiv preprint arXiv:2410.08662},
	year={2024}
}

@article{shanmukha2002metric,
	title={Metric dimension of wheels},
	author={Shanmukha, B and Sooryanarayana, B and Harinath, KS},
	journal={Far East J. Appl. Math},
	volume={8},
	number={3},
	pages={217--229},
	year={2002}
}

@article{Slater1975,
	title={Leaves of trees},
	author={Slater, Peter J},
	journal={Congr. Numer},
	volume={14},
	number={549-559},
	pages={37},
	year={1975}
}

@article{tillquist2019low,
	date = {2019/07/01},
	author={Tillquist, Richard C and Lladser, Manuel E},
	doi = {10.1007/s00285-019-01348-1},
	id = {Tillquist2019},
	isbn = {1432-1416},
	journal = {J. Math. Biol.},
	number = {1},
	pages = {1--29},
	title = {Low-dimensional representation of genomic sequences},
	volume = {79},
	year = {2019}}

@article{kuziak2010corrections,
	title={Corrections to the article" The metric dimension of graph with pendant edges"[Journal of Combinatorial Mathematics and Combinatorial Computing, 65 (2008) 139--145]},
	author={Kuziak, D and Rodriguez-Velazquez, JA and Yero, IG},
	journal={arXiv preprint arXiv:1010.1784},
	year={2010}
}

@article{kuziak2017computing,
	author = {Dorota Kuziak, Juan A. Rodr{\'\i}guez-Vel{\'a}zquez, Ismael G. Yero},
	title = {Computing the Metric Dimension of a Graph from Primary Subgraphs},
	journal = {Discuss. Math. Graph Theory},
	language = {eng},
	number = {1},
	pages = {273-293},
	volume = {37},
	year = {2017},
}

@article{yero2011metric,
	title = {On the metric dimension of corona product graphs},
	journal = {Comput. Appl. Math.},
	volume = {61},
	number = {9},
	pages = {2793-2798},
	year = {2011},
	issn = {0898-1221},
	doi = {10.1016/j.camwa.2011.03.046},
	author = {I.G. Yero and D. Kuziak and J.A. Rodríguez-Velázquez}
}

@article{Erdos,
	title={On two problems of information theory},
	author={Erd{\H{o}}s, P{\'a}l and R{\'e}nyi, Alfr{\'e}d},
	journal={A Magyar Tudom{\'a}nyos Akad{\'e}mia Matematikai Kutat{\'o} Int{\'e}zet{\'e}nek K{\"o}zlem{\'e}nyei},
	volume={8},
	number={1-2},
	pages={229--243},
	year={1963},
	publisher={Akad{\'e}miai Kiad{\'o}}
}

@article{pav,
	title={Using graph theory to analyze biological networks},
	author={Pavlopoulos, Georgios A and Secrier, Maria and Moschopoulos, Charalampos N and Soldatos, Theodoros G and Kossida, Sophia and Aerts, Jan and Schneider, Reinhard and Bagos, Pantelis G},
	journal={BioData Min.},
	volume={4},
	pages={1--27},
	year={2011},
	publisher={Springer},
	doi={10.1186/1756-0381-4-10}
}

@article{dav,
	title = {Response prediction of antidepressants: Using graph theory tools for brain network connectivity analysis},
	journal = {Biomed. Signal Process. Control.},
	volume = {103},
	pages = {107362},
	year = {2025},
	issn = {1746-8094},
	doi = {10.1016/j.bspc.2024.107362},
	author = {Akbar Davoodi and Martin Holeňa and Martin Brunovský and Aditi Kathpalia and Jaroslav Hlinka and Martin Bareš and Milan Paluš}
}

@article{wag,
	title = {Assessing the vulnerability of supply chains using graph theory},
	journal = {Int. J. Prod. Econ.},
	volume = {126},
	number = {1},
	pages = {121-129},
	year = {2010},
	issn = {0925-5273},
	doi = {10.1016/j.ijpe.2009.10.007},
	author = {Stephan M. Wagner and Nikrouz Neshat}
}

@article{circulant,
	title = {On the metric dimension of circulant and Harary graphs},
	journal = {Appl Math Comput.},
	volume = {248},
	pages = {47-54},
	year = {2014},
	issn = {0096-3003},
	doi = {10.1016/j.amc.2014.09.045},
	author = {Cyriac Grigorious and Paul Manuel and Mirka Miller and Bharati Rajan and Sudeep Stephen}
}

@article{fractal,
	title = {Redefining fractal cubic networks and determining their metric dimension and fault-tolerant metric dimension},
	journal = {Appl Math Comput.},
	volume = {452},
	pages = {128037},
	year = {2023},
	issn = {0096-3003},
	doi = {10.1016/j.amc.2023.128037},
	author = {M. Arulperumjothi and Sandi Klavžar and S. Prabhu},
}

@article{starfan,
	title={Metric dimension of star fan graph},
	author={Prabhu, S and Jeba, D Sagaya Rani and Stephen, Sudeep},
	journal={Scientific Reports},
	volume={15},
	number={1},
	pages={102},
	year={2025},
	publisher={Nature Publishing Group UK London}
}

@article{carbon,
	title={Metric basis and metric dimension of 1-pentagonal carbon nanocone networks},
	author={Hussain, Zafar and Munir, Mobeen and Ahmad, Ashfaq and Chaudhary, Maqbool and Alam Khan, Junaid and Ahmed, Imtiaz},
	journal={Scientific Reports},
	volume={10},
	number={1},
	pages={19687},
	year={2020},
	publisher={Nature Publishing Group UK London}
}

@book{hammack2011handbook,
	title={Handbook of product graphs},
	author={Hammack, Richard H and Imrich, Wilfried and Klav{\v{z}}ar, Sandi},
	volume={2},
	year={2011},
	publisher = {CRC Press},
	address   = {Boca Raton},
	doi={10.1201/b10959}
}

@article{line,
	title = {On the metric dimension of line graphs},
	journal = {Discrete Appl Math},
	volume = {161},
	number = {6},
	pages = {802-805},
	year = {2013},
	issn = {0166-218X},
	doi = {10.1016/j.dam.2012.10.018},
	author = {Min Feng and Min Xu and Kaishun Wang}
}

@article{partitiondimension,
	title={The partition dimension of a graph},
	author={Chartrand, Gary and Salehi, Ebrahim and Zhang, Ping},
	journal={Aequationes mathematicae},
	volume={59},
	number={1},
	pages={45--54},
	year={2000},
	publisher={Springer},
	doi={10.1007/PL00000127}
}

@article{locating-domination,
	title={On locating-domination in graphs},
	author={Chellali, Mustapha and Mimouni, Malika and Slater, Peter},
	journal={Discussiones Mathematicae Graph Theory},
	volume={30},
	number={2},
	pages={223--235},
	year={2010},
	publisher={Uniwersytet Zielonog{\'o}rski. Wydzia{\l} Matematyki, Informatyki i Ekonometrii}
}

@article{lcc,
	title = {Clique coverings and claw-free graphs},
	journal = {Eur. J. Comb.},
	volume = {88},
	pages = {103114},
	year = {2020},
	issn = {0195-6698},
	doi = {10.1016/j.ejc.2020.103114},
	author = {Csilla Bujtás and Akbar Davoodi and Ervin Győri and Zsolt Tuza}
}

@article{triangle,
	title = {Covering the edges of a graph with triangles},
	journal = {	Discrete Math.},
	volume = {348},
	number = {1},
	pages = {114226},
	year = {2025},
	issn = {0012-365X},
	doi = {10.1016/j.disc.2024.114226},
	author = {Csilla Bujtás and Akbar Davoodi and Laihao Ding and Ervin Győri and Zsolt Tuza and Donglei Yang}
}

@article{scp,
	title = {Pairwise balanced designs and sigma clique partitions},
	journal = {	Discrete Math.},
	volume = {339},
	number = {5},
	pages = {1450-1458},
	year = {2016},
	issn = {0012-365X},
	doi = {10.1016/j.disc.2015.09.008},
	author = {Akbar Davoodi and Ramin Javadi and Behnaz Omoomi}
}

@article{scc,
	title={Edge clique covering sum of graphs},
	author={Davoodi, Akbar and Javadi, Ramin and Omoomi, Behnaz},
	journal={Acta Mathematica Hungarica},
	volume={149},
	pages={82--91},
	year={2016},
	publisher={Springer},
	doi={10.1007/s10474-016-0586-1}
}

@inproceedings{pullman,
	title={Clique coverings of graphs—a survey},
	author={Pullman, Norman J},
	booktitle={Combinatorial Mathematics X: Proceedings of the Conference held in Adelaide, Australia, August 23--27, 1982},
	pages={72--85},
	year={2006},
	organization={Springer},
	doi={10.1007/BFb0071509}
}

@article{Farooq2025,
	doi = {10.1371/journal.pone.0316376},
	author = {Farooq, Umar AND Abbas, Wasim AND Chaudhry, Faryal AND Azeem, Muhammad AND Almohsen, Bandar},
	journal = {PLOS ONE},
	publisher = {Public Library of Science},
	title = {Exploring metric dimension of nanosheets, nanotubes, and nanotori of SiO2},
	year = {2025},
	month = {03},
	volume = {20},
	pages = {1-22},
	number = {3},
	
}

@article{Liqin2023,
	doi = {10.1371/journal.pone.0290411},
	author = {Liqin, Liu AND Shahzad, Khurram AND Rauf, Abdul AND Tchier, Fairouz AND Aslam, Adnan},
	journal = {PLOS ONE},
	publisher = {Public Library of Science},
	title = {Metric and fault-tolerant metric dimension for GeSbTe superlattice chemical structure},
	year = {2023},
	month = {11},
	volume = {18},
	pages = {1-14},
	abstract = {The concept of metric dimension has many applications, including optimizing sensor placement in networks and identifying influential persons in social networks, which aids in effective resource allocation and focused interventions; finding the source of a spread in an arrangement; canonically labeling graphs; and inserting typical information in low-dimensional Euclidean spaces. In a graph G, the set S⊆V(G) of minimum vertices from which all other verticescan be uniquely determined by the distances to the vertices in S is called the resolving set. The cardinality of the resolving set is called the metric dimension. The set S is called fault-tolerant resolving set if S\{v} is still a resolving set of G. The minimum cardinality of such a set S is called fault-tolerant metric dimension of G. GeSbTe super lattice is the latest chemical compound to have electronic material that is capable of non-volatile storing phase change memories with minimum energy usage. In this work, we calculate the resolving set (fault tolerant resolving set) to find the metric dimension(fault-tolerant metric dimension) for the molecular structure of the GeSbTe lattice. The results may be useful in comparing network structure and categorizing the structure of the GeSbTe lattice.}

@article{zhang2023factorgraph,
	title = {A multi-sensor fusion positioning approach for indoor mobile robot using factor graph},
	journal = {Measurement},
	volume = {216},
	pages = {112926},
	year = {2023},
	issn = {0263-2241},
	doi = {10.1016/j.measurement.2023.112926},
	author = {Liyang Zhang and Xingyu Wu and Rui Gao and Lei Pan and Qian Zhang}
}

@article{huang2025graphpool,
	title = {Sparse regularized graph pooling network optimal sensor placement method for diesel engine vibration fault perception system},
	journal = {Measurement},
	volume = {242},
	pages = {115830},
	year = {2025},
	issn = {0263-2241},
	doi = {10.1016/j.measurement.2024.115830},
	author = {Anzheng Huang and Zhiwei Mao and Fengchun Liu and Jinjie Zhang and Xiangxin Kong and Zhinong Jiang}
}

@article{mermer2025entropy,
	title = {A novel entropy calculation approach for sensor placement optimization for a satellite thermal control system},
	journal = {Measurement},
	volume = {249},
	pages = {116963},
	year = {2025},
	issn = {0263-2241},
	doi = {10.1016/j.measurement.2025.116963},
	author = {Erdinç Mermer and Rahmi Ünal}
}

@article{Transp,
	author = {Sybil Derrible and Christopher Kennedy},
	title = {Applications of Graph Theory and Network Science to Transit Network Design},
	journal = {Transp. Rev.},
	volume = {31},
	number = {4},
	pages = {495--519},
	year = {2011},
	publisher = {Routledge},
	doi = {10.1080/01441647.2010.543709}
}

\appendix

\end{document}